\documentclass[10pt,a4paper]{article}
\usepackage[utf8]{inputenc}
\usepackage{amsmath}
\usepackage{amsfonts}
\usepackage{amssymb}
\usepackage{amsthm}
\usepackage{enumitem}
\usepackage{color}
\usepackage{hyperref}
\usepackage{MnSymbol}
\usepackage{graphicx}

\usepackage[a4paper, textwidth=13cm]{geometry}

\newtheorem{theorem}{Theorem}
\newtheorem{remark}[theorem]{Remark}

\newtheorem{lemma}[theorem]{Lemma}
\newtheorem{proposition}[theorem]{Proposition}
\newtheorem{notation}[theorem]{Notation}

\newcommand{\ce}{c_{\epsilon}}
\newcommand{\cepl}{c^+_{\epsilon}}
\newcommand{\cemin}{c^-_{\epsilon}}

\newcommand{\cepm}{c^{\pm}_{\epsilon}}

\newcommand{\cem}{c^M_{\epsilon}}
\newcommand{\tcepm}{\widetilde{c}^{\pm}_{\epsilon}}

\newcommand{\dem}{D^M_{\epsilon}}

\newcommand{\fepm}{f^{\pm}_{\epsilon}}

\newcommand{\gem}{g^M_{\epsilon}}

\newcommand{\pepm}{\phi_{\epsilon}^{\pm}}
\newcommand{\pem}{\phi_{\epsilon}^M}
\newcommand{\hop}{H^1_{\#}}
\newcommand{\app}{\mathrm{app}}
\newcommand{\bl}{\mathrm{bl}}
\newcommand{\bpem}{\overline{\phi}_{\epsilon}^M}

\newcommand{\x}{\bar{x}}
\newcommand{\y}{\bar{y}}

\newcommand{\fxe}{\frac{x}{\epsilon}}

\renewcommand{\oe}{\Omega_{\epsilon}}
\newcommand{\oem}{\Omega_{\epsilon}^M}

\newcommand{\oepl}{\Omega_{\epsilon}^+}
\newcommand{\oemi}{\Omega_{\epsilon}^-}
\newcommand{\sepm}{S_{\epsilon}^{\pm}}
\newcommand{\sepl}{S_{\epsilon}^+}
\newcommand{\semi}{S_{\epsilon}^-}
\newcommand{\oepm}{\Omega_{\epsilon}^{\pm}}

\newcommand{\foe}{\frac{1}{\epsilon}}
\newcommand{\N}{\mathbb{N}}
\newcommand{\R}{\mathbb{R}}

\newcommand{\ie}{i.\,e., }

\newcommand{\loc}{\mathrm{loc}}
\renewcommand{\u}{\bar{u}}

\newcommand{\eg}{e.\,g.}

\newcommand{\zi}{Z_{\infty}}

\title{Correctors and error estimates for reaction-diffusion processes through thin heterogeneous layers in case of homogenized equations with interface diffusion}

\author{M. Gahn\thanks{Interdisciplinary Center for Scientific Computing, University of Heidelberg,   Heidelberg, Germany. \textit{Mail: markus.gahn@iwr.uni-heidelberg.de}}
\and W. J\"ager\thanks{Interdisciplinary Center for Scientific Computing, University of Heidelberg,   Heidelberg, Germany. \textit{Mail: wjaeger@iwr.uni-heidelberg.de}}
 \and  M. Neuss-Radu\thanks{Department of Mathematics, Friedrich-Alexander-Universit\"at Erlangen-N\"urnberg,  Erlangen, Germany. \textit{Mail: maria.neuss-radu@math.fau.de }} }
\date{}

\begin{document}

\maketitle

\begin{abstract}
In this paper, we construct approximations of the microscopic solution of a nonlinear reaction--diffusion equation in a domain consisting of two bulk-domains, which are separated by a thin  layer with a periodic heterogeneous structure. The size of the heterogeneities and thickness of the layer are of order $\epsilon$, where the parameter $\epsilon$ is small compared to the length scale of the whole domain. In the limit $\epsilon \to 0$, when the thin layer reduces to an interface $\Sigma$ separating two bulk domains, a macroscopic model with effective interface conditions across $\Sigma$ is obtained.  Our approximations are obtained by adding corrector terms to the macroscopic solution, which take into account the oscillations in the thin layer and the coupling conditions between the layer and the bulk domains. To validate these approximations, we prove error estimates with respect to $\epsilon$. Our approximations are constructed in two steps leading to error estimates of order $\epsilon^{\frac12}$ and $\epsilon$ in the $H^1$-norm.
\end{abstract}



\section{Introduction}

Problems including reactive transport processes through thin layers with a heterogeneous structure play an important role in many applications, especially from biosciences, medical sciences, geosciences, and material sciences. We mention here as an example the physiological processes in blood vessels, where the endothelial layer, separating the lumen (region occupied by blood flow) from the vessel wall, mediates and controls the exchange between these two regions. In \cite{Telma} a detailed model for processes at the endothelium is given by using phenomenologically derived effective interface laws. However, multi-scale techniques for the rigorous derivation of such laws starting from microscopic models, the study of their validity range and of the accuracy of the approximations are urgently needed. The  techniques developed in this paper give an important contribution to this field, even though our model problem is limited to reaction--diffusion processes, and thus omitting further aspects like e.g. advective transport or mechano-chemical interactions.  

In this paper, we consider a nonlinear reaction--diffusion equation in a domain $\oe$ consisting of two bulk-domains $\oepl$ and $\oemi$ which are separated by a thin layer $\oem$ with a periodic heterogeneous structure. The thickness of the thin layer as well as the period of the heterogeneities are of order $\epsilon>0$, where the parameter $\epsilon $ is small compared to the length scale of the whole domain $\oe$. Across the interfaces $\sepm$ between the bulk-domains $\oepm$ and the thin layer $\oem$ we assume continuity of the solution and its normal flux. The numerical computation of the solution to this type of problems faces a high complexity. Therefore, 
we construct approximations of the microscopic solution which can be calculated with less numerical effort, and prove error estimates between the approximation and the microscopic solution with respect to the scaling parameter $\epsilon$. Such error estimates are important for the justification of the approximation as well as for predictions about its accuracy. There is a rich literature on problems in thin domains with applications in solid mechanics, wave diffraction, porous media and so on, where we have to mention the monographs \cite{SanchezPalencia1980} and \cite{BakhvalovPanasenko1989}.   Results about the derivation of effective transmission conditions by multi-scale techniques for domains separated by thin heterogeneous layers can be found e.g. in \cite{BourgeatGipoulouxMarusicPalokaModellingUndergroundWaste,BourgeatMarusicPaloka2005,ConcaI1987,ConcaII1987,GahnEffectiveTransmissionContinuous,GahnNeussRaduKnabner2018a,Orlik2016,Orlik2017,Moussa_2012,NeussJaeger_EffectiveTransmission,PopBogersKumar2017}. However, error estimates for thin heterogeneous layers coupled to bulk-domains have hardly been considered in literature so far.
 We mention here the paper \cite{Panasenko1981}, where an elliptic problem in a domain including a thin heterogeneous layer with thickness of order $\epsilon$ is treated for different scalings for the diffusion coefficients with respect to $\epsilon$. More precisely, based on the Bakhvalov-ansatz the author derives higher order asymptotic approximations for the microscopic solution and derives error estimates with respect to $\epsilon$. This method uses high regularity results for the microscopic and the macroscopic solutions, and therefore for the data. The asymptotic expansion is formally extended to nonlinear and nonstationary problems.
In  \cite{AllaireII1991,bourgeat1997ecoulement,Marusic1998}  the asymptotic behavior of a fluid flow through a filter formed by $\epsilon$-periodic distribution of obstacles of size $\epsilon^\beta$ distributed on a hypersurface is considered and error estimates are proved. In \cite{JaegerOleinikShaposhnikova} a similar geometrical setting was considered and the Poisson equation with mixed boundary conditions on the boundary of the obstacles was studied. In \cite{Shaposhnikova2001} a Neumann problem in a domain containing a thin filter consisting of periodically distributed channels is considered.

In  \cite{GahnEffectiveTransmissionContinuous,GahnNeussRaduKnabner2018a,NeussJaeger_EffectiveTransmission} reaction-diffusion equations through thin heterogeneous layers were studied for different scalings in the thin layer, and effective models for $\epsilon \to 0$  were derived. In the singular limit, the thin layer reduces to a $(n-1)$-dimensional interface $\Sigma$ separating the bulk-regions $\Omega^+$ and $\Omega^-$. In these bulk-domains the evolution of the limit problem carries the same structure as in the microscopic problem, whereas at the interface $\Sigma$ effective interface laws emerge. 
Of particular importance is the choice of the 
scaling of the coefficients, especially in the equations in the thin layer. The scaling highly influences the structure of the macroscopic model 
and  depends  on the particular application. 

In the present paper, we consider a specific scaling from \cite{GahnEffectiveTransmissionContinuous} leading to a reaction-diffusion equation on the interface $\Sigma$ in the limit $\epsilon \to 0$.  This effective interface condition for the macroscopic model is similar to a result in \cite{Panasenko1981} (see the interface condition (1.10)). However, in our case we consider a different scaling for the microscopic equation, and a nonlinear and nonstationary problem. Further, we consider a scaling for the reaction term in the thin layer which gives an additional contribution in the effective interface condition.
For this situation, we investigate the quality of the approximation of the microscopic solution by means of the macroscopic one. 
In general, we cannot expect strong convergence of the gradients or high-order error estimates with respect to $\epsilon$. For such results we have to add additional corrector terms to the macroscopic solution which take into account the oscillations in the thin layer and also the coupling conditions between the bulk-regions and the layer. The construction of the approximations is made in two steps. Firstly, we add to the macroscopic solution in the thin layer a corrector of order $\epsilon$, which carries information about the oscillations in the layer. This leads to error estimates of order $\epsilon^{\frac12}$ in the $H^1$-norms, see Theorem \ref{MainTheoremFirstOrderApproximation}. To obtain a better estimate, in a second step,  we add a corrector term of first order to the macroscopic solutions in the bulk-domains (which equilibrates the discontinuity of the approximation across $\sepm$) and an additional second order corrector to the macroscopic solution in the layer (which equilibrates the discontinuity of the normal fluxes of the approximation across $\sepm$). 
This strategy of stepwise building up the correctors has also been used e.g. in \cite{JaegerMikelic1996}.
The resulting approximation leads to an error estimate of order $\epsilon$ in the $H^1$-norms, see Theorem \ref{MainTheoremSecondOrderApproximation}.

The major challenge in our paper is the simultaneous scale transition for the thickness of the layer and the periodic heterogeneous structure within the layer, as well as the coupling between the bulk-domains and the thin layer, where we additionally have to take into account different kinds of scaling. 
In this context, the presence of nonlinear reaction terms creates additional difficulties. These specific features of our problem are also reflected by differences in the form of the correctors in the two regions (bulk domains and thin layer):
The order of the corrector terms with respect to $\epsilon$ is different in the two regions. 
Furthermore, in the layer, the correctors are obtained by products of the derivatives of the macroscopic solution and solutions of suitable cell problems on a bounded reference element $Z$, whereas in the bulk-domains the correctors include solution of boundary layer problems in infinite stripes. 

To justify the determined approximations, we prove error estimates. Roughly speaking, we apply the microscopic differential operator to the microscopic solution as well as to the approximations and subtract the terms from each other. 
To estimate the arising terms on the right hand side, the main idea is to represent solenoidal vector fields by the divergence of skew symmetric matrices. Integration by parts then yields an additional factor $\epsilon$ which can be exploited for the error estimates. This approach has been previously used in \cite{JikovKozlovOleinik1994} for vector fields defined on the standard periodicity cell and periodic in all directions, and in \cite{neuss2001boundary} for boundary layers. In our case  the situation is more difficult and we have to construct skew-symmetric matrices adapted to the structure of our problem. More precisely, these matrices have to be such that boundary terms which occur at the interfaces between the bulk domains and the thin layer vanish.  

Our paper is structured as follows: In Section \ref{SectionMicroscopicModel}
we present the microscopic model, the assumptions on the data as well as the a priori estimates for the microscopic solution. In Section \ref{SectionMainResults} we give the general form of the two approximations for the microscopic solution and state the corresponding error estimates. The macroscopic model and the higher order correctors are introduced in  Sections \ref{SectionZerothOrderProblem} and \ref{SectionAuxiliaryProblems} respectively. The proof for the error estimates is given in Section \ref{SectionErrorEstimates}. We conclude the paper with a short appendix about the regularity of the solution to the macroscopic problem.

\subsection{Original contributions}

In the literature, there are several results dealing with singular limits for reactive transport processes through  thin layers with heterogeneous structures. However, results including error estimates with respect to $\epsilon$ seem to be rather rare, especially with regard to nonlinear problems.  In this paper we construct asymptotic approximations for the microscopic solution of a semilinear reaction-diffusion problem in a domain including a thin heterogeneous layer, and derive error estimates with respect to the scaling parameter $\epsilon$. Besides providing control on the quality of the approximation, the correctors constructed by means of boundary layers have the suitable complexity to be used for numerical simulations (reducing the high complexity induced by the microstructure). The main contributions of our paper are:
\begin{enumerate}[label = -]
\item the construction of approximations for the microscopic solution including correctors and boundary layers adapted to the scaling in the microscopic problem and the microscopic structure of the thin layer, as well as to the transmission conditions at the bulk-layer interfaces,
\item the derivation of  error estimates of order $\sqrt{\epsilon}$ and $\epsilon$ in the context of nonlinear problems which combine the classical approach of homogenization for microscopic structures with a singular limit approach,
 \item proving of error estimates under low regularity assumptions on the data and therefore low regularity for the microscopic and the macroscopic solution,
\item extending the approach for the representation of solenoidal vector fields by the divergence of skew symmetric matrices to boundary layers involving transmission conditions.
\end{enumerate}

\section{The microscopic model}
\label{SectionMicroscopicModel}

We consider the domain $\oe := \Sigma \times (-\epsilon -H,H + \epsilon) \subset \R^n$ with fixed $H \in \N$, $n\geq 2$, and $\Sigma =(0,l_1)\times \ldots, \times (0,l_{n-1}) \subset \R^{n-1}$
with $l = (l_1,\ldots,l_{n-1}) \in \N^{n-1}$. Further, let $\epsilon >0$ be a sequence with $\epsilon^{-1} \in \N$.  The set $\oe$ consists of  three subdomains, see Figure \ref{FigurMikroskopischesGebietDuenneSchicht},  given by
\begin{align*}
\oepl &:= \Sigma \times (\epsilon,H + \epsilon),
\\
\oem &:= \Sigma \times (-\epsilon,\epsilon),
\\
\oemi &:= \Sigma \times (-\epsilon -H,-\epsilon).
\end{align*}

The domains $\oepm$ and $\oem$ are separated by an interface $\sepm$, \ie
\begin{align*}
\sepl:= \Sigma \times \{\epsilon\} \quad \mbox{ and } \quad \semi:= \Sigma \times \{-\epsilon\},
\end{align*}
hence, we have $\oe = \oepl \cup \oemi \cup \oem \cup \sepl \cup \semi$.

\begin{figure}[h]
\centering
\includegraphics[scale=0.18]{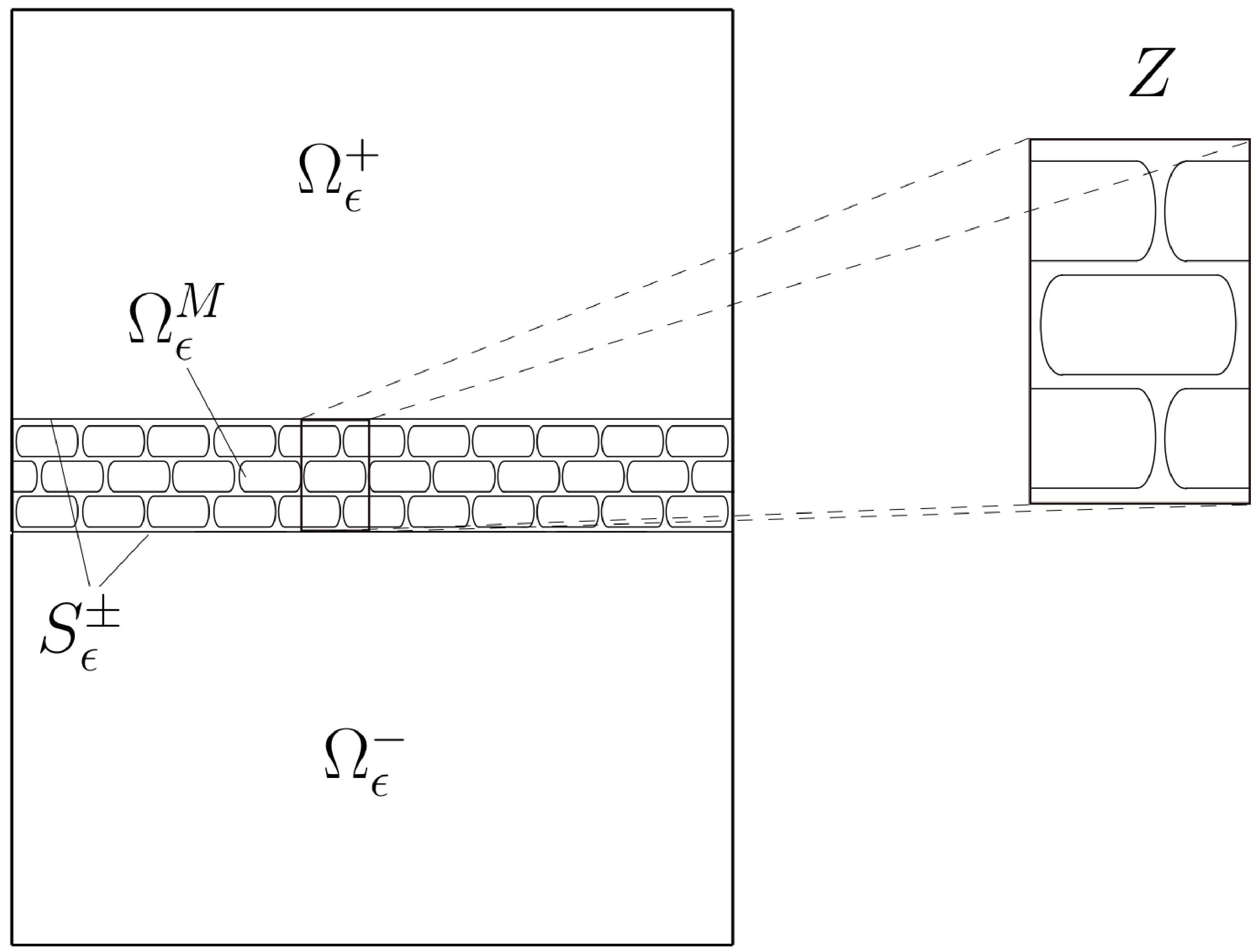}
 \caption{The microscopic domain containing the thin layer $\oe$ with periodic structure for $n=2$. The heterogeneous structure for the thin layer is modeled by the diffusion coefficient $\dem$, see Assumption \ref{VoraussetzungenDuenneSchichtDiffusionskoeffizienten}.}\label{FigurMikroskopischesGebietDuenneSchicht}
\end{figure}

 As mentioned above, for $\epsilon \to 0$ the membrane $\oem$ reduces to an interface $\Sigma \times \{0\}$, which we also denote by $\Sigma$ suppressing the $n$-th component,  and we define 
\begin{align*}
\Omega^+:=\Sigma \times (0,H) \quad \mbox{ and } \quad \Omega^-:=\Sigma\times (-H,0),
\end{align*}
and $\Omega:= \Omega^+ \cup \Sigma \cup \Omega^- = \Sigma \times (-H,H)$.
The microscopic structure within the thin layer $\oem$ can be described by shifted and scaled reference elements. We define
\begin{align*}
Y^k &:= (0,1)^k \quad \mbox{ for } k \in \N,
\\
Z &:= Y^{n-1} \times (-1,1).
\end{align*}
We denote the upper and lower boundary of $Z$ by
\begin{align*}
S^+:= Y^{n-1} \times \{1\} \quad \mbox{and} \quad S^-:= Y^{n-1} \times \{-1\} .
\end{align*}

Now, we consider a reaction-diffusion problem describing the evolution of a substance with  concentration $c_\epsilon$ in the domain $\Omega_\epsilon$. Let $\ce = (\cepl,\cem,\cemin)$ with $\cepm: (0,T)\times \oepm \rightarrow \R$ and $\cem: (0,T)\times \oem \rightarrow \R$ be a solution of the following microscopic problem
\begin{subequations}\label{MicroscopicModel}
\begin{align}
\begin{aligned}\label{PDE_ReacDiffEquation}
\partial_t \cepm - \nabla \cdot \left( D^{\pm} \nabla \cepm \right)  &= \fepm(t,x,\cepm) &\mbox{ in }& (0,T) \times \oepm,
\\
\foe \partial_t \cem - \foe \nabla \cdot \left(\dem \nabla  \cem \right)  &= \frac{1}{\epsilon} \gem(t,x,\cem)&\mbox{ in }& (0,T)\times \oem,
\end{aligned}
\end{align}
together with continuous transmission conditions across the interfaces $\sepm$
\begin{align}
\begin{aligned}
\cepm &= \cem &\mbox{ on }& (0,T) \times \sepm,
\\
- D^{\pm}\nabla \cepm\cdot \nu &= -\foe \dem\nabla \cem  \cdot \nu &\mbox{ on }& (0,T)\times \sepm,
\end{aligned}
\end{align}
where $\nu$ denotes the outer unit normal on $\sepm$ with respect to $\oem$, and
the outer boundary conditions
\begin{align}
\begin{aligned}
 D^{\pm} \nabla \cepm \cdot \nu &= 0 &\mbox{ on }& (0,T) \times \Sigma \times \{\pm(H + \epsilon)\},
\\
\ce \mbox{ is }&\Sigma\mbox{-periodic},
\end{aligned}
\end{align}
and the initial condition
\begin{align}
\begin{aligned}
\label{InitialConditions}
\ce(0,x)= c_{\epsilon}^0(x) := \begin{cases}
c_{\epsilon}^{0,\pm}\left(\x, x_n \mp \epsilon\right) &\mbox{ for } x \in \oepm,
\\
c^{0,M}_{\epsilon}\left(x\right) &\mbox{ for } x \in \oem.
\end{cases}
\end{aligned}
\end{align}
\end{subequations}
Here, $D^\pm$ and $D^M$ denote the diffusion coefficients of the substance in the bulk regions and in the thin layer respectively. The functions $f_\epsilon^\pm$ and $g_\epsilon^M$ model nonlinear reaction terms. Note that the heterogeneous structure of the thin layer is modeled by the oscillating diffusion coefficient $D^M_\epsilon$ and by the dependence of the reaction term $g_\epsilon^M$ on the microscopic variable $y= \frac{x}{\epsilon}$.

In the following, we define for an arbitrary interval $I \subset \R$ and a rectangle $W\subset \R^{n-1}$ the Sobolev space of $W$-periodic functions, \ie the space of functions which are $W$-periodic with respect to first $n-1$ components: 
\begin{align*}
H^k_{\#}(W \times I):= \left\{ u \in H^k(W\times I)\, : \, u \mbox{ is } W\mbox{-periodic}\right\},
\end{align*}
for $k\in \N$. In a similar way we also use the index $\#$ for  spaces with $Y^{n-1}$-periodic functions, \eg,  $C^{0,1}_{\#}\left(\overline{Z}\right):= \left\{u \in C^{0,1}\left(\overline{Z}\right) \, : \, u \mbox{ is } Y^{n-1}\mbox{-periodic}\right\}$.

\vspace{1em}

\noindent\textbf{Assumptions on the data:}

\begin{enumerate}
[label = (A\arabic*)]
\item \label{VoraussetzungenDuenneSchichtDiffusionskoeffizienten} It holds $D^{\pm}  \in \R^{n\times n} $ is symmetric and positive-definite,  and $\dem(x)=D^M\left(\fxe\right)$ with  $D^M \in C_{\#}^{0,1}\big(\overline{Z}\big)^{n\times n}$. Further,  $D^M$ is symmetric and coercive, \ie there exists a constant $c_0>0$ such that 
\begin{align*}
D^M(y)\xi \cdot \xi &\geq c_0 \|\xi\|^2, \quad \mbox{for all } y \in Z,\, \xi \in \R^n.
\end{align*}
\item \label{VoraussetzungenDuenneSchichtReaktionskinetikBulk}
We have $\fepm(t,x,z) = f^{\pm}\left(t,\fxe,z\right)$ with $f^{\pm}: [0,T]  \times \R^n \times \R \rightarrow \R$ is continuous,   $Y^n$-periodic with respect to the second variable, and uniformly Lipschitz continuous with respect to the third variable.
The uniform Lipschitz condition ensures the estimate
\begin{align*}
|f^{\pm}(t,y,z)| \le C\big(1 + |z|\big) \quad \mbox{ for all } (t,y.z) \in [0,T] \times \R^n\times \R.
\end{align*}
\item \label{VoraussetzungenDuenneSchichtReaktionskinetikMembran}
We have $\gem(t,y,z) = g^M\left(t,\fxe,z\right)$ with  $g^M: [0,T] \times \R^{n-1} \times [-1,1] \times \R \rightarrow \R$ is continuous, uniformly Lipschitz continuous with respect to the last variable, and $Y^{n-1}$-periodic with respect to the second variable. As above, we have
\begin{align*}
|g^M(t,y,z)|\le C\big(1 + |z|\big) \, \mbox{ for all } (t,y,z)\in [0,T]\times \R^{n-1} \times [-1,1] \times \R.
\end{align*}
Note that here, the variable $y$ describes the independent variable corresponding to the standard cell $Z$, and the variable $z$ corresponds to the dependent variable modeling the concentration $c^M_\epsilon$.

\item \label{VoraussetzungenDuenneSchichtAnfangsbedingungen} For the initial functions, we assume $c_{\epsilon}^{0,\pm} \in H_{\#}^1(\Omega^{\pm})$  and $c_{\epsilon}^{0,M}\in H_{\#}^1(\oem )$ with $c_{\epsilon}^{0,\pm}|_{\Sigma} = c_{\epsilon}^{0,M}|_{\sepm}$, and they fulfill the following estimate
\begin{align*}
\big\|c_{\epsilon}^{0,\pm} \big\|_{H^1(\Omega^{\pm})} + \frac{1}{\sqrt{\epsilon}} \big\|c_{\epsilon}^{0,M} \big\|_{H^1(\oem)} \le C.
\end{align*}
Further we assume that there exist $c^{0,\pm} \in H_{\#}^1(\Omega^{\pm})$ and $c^{0,M} \in H_{\#}^1(\Sigma)$ with $c^{0,\pm}|_{\Sigma} = c^{0,M}$, such that
\begin{align}\label{ErrorEstimateInitialValues}
\big\|c_{\epsilon}^{0,\pm} - c^{0,\pm}\big\|_{L^2(\Omega^{\pm})} + \frac{1}{\sqrt{\epsilon}} \big\|c_{\epsilon}^{0,M} - c^{0,M} \big\|_{L^2(\oem)} \le C\epsilon.
\end{align}
\end{enumerate}

\noindent\textbf{Weak formulation}
A function $\ce \in L^2((0,T),\hop(\oe))\cap H^1((0,T),L^2(\oe))$ is called a weak solution of Problem $\eqref{PDE_ReacDiffEquation}$, if for all $\phi \in \hop(\oe)$,  and almost every $t\in (0,T)$ it holds that
\begin{align}
\label{VE_OmegaEps}
\begin{aligned}
 \sum_{\pm} & \left[\int_{\oepm}\partial_t \cepm \phi dx + \int_{\oepm}  D^{\pm} \nabla \cepm \cdot \nabla \phi dx  \right] 
+ \frac{1}{\epsilon}\int_{\oem}\partial_t \cem \phi dx 
\\
&+ \foe\int_{\oem} \dem \nabla \cem \cdot \nabla \phi dx
= \sum_{\pm} \int_{\oepm} \fepm \big(\cepm\big) \phi dx + \foe \int_{\oem} \gem\big(\cem\big) \phi dx
\end{aligned}
\end{align}
together with the initial condition $\eqref{InitialConditions}$.

To work on the fixed domains $\Omega^{\pm}$, we shift the $\epsilon$-dependent domains $\oepm$ to the fixed domains $\Omega^{\pm}$, and give an equivalent formulation for a weak solution on the fixed domains $\Omega^{\pm}$. We define
\begin{align*}
\tcepm: (0,T)\times \Omega^{\pm} \rightarrow \R,\quad \tcepm(t,x):= \cepm(t,\x,x_n \pm \epsilon).
\end{align*}
Then $\ce$ is a weak solution of Problem $\eqref{PDE_ReacDiffEquation}$, iff $\tcepm \in L^2((0,T),\hop(\Omega^{\pm})) \cap H^1((0,T),L^2(\Omega^{\pm}))$ and $\cem \in L^2((0,T),\hop(\oem))\cap H^1((0,T),L^2(\oem))$ with $\cem|_{\sepm} = \tcepm|_{\Sigma}$, and for all $\pepm \in \hop(\Omega^{\pm})$ and $\pem\in \hop(\oem)$ with $\pepm|_{\Sigma} = \pem|_{\sepm}$ it holds that
\begin{align}
\begin{aligned}\label{VE_Omega}
 \sum_{\pm} & \left[\int_{\Omega^{\pm}}\partial_t \tcepm \pepm dx + \int_{\Omega^{\pm}}  D^{\pm} \nabla \tcepm \cdot \nabla \pepm dx  \right] 
+ \frac{1}{\epsilon}\int_{\oem}\partial_t \cem \pem dx 
\\
+& \foe \int_{\oem} \dem \nabla \cem \cdot \nabla \pem dx
= \sum_{\pm} \int_{\Omega^{\pm}} \fepm \big(\tcepm\big) \pepm dx + \foe \int_{\oem} \gem\big(\cem\big) \pem dx
\end{aligned}
\end{align}
together with the initial condition
\begin{align*}
\tilde{c}_{\epsilon}(0,x)= \begin{cases}
c_{\epsilon}^{0,\pm}(x) &\mbox{ for } x \in \Omega^{\pm},
\\
c_{\epsilon}^{0,M}(x) &\mbox{ for } x \in \oem.
\end{cases}
\end{align*}

\begin{notation}
In the following, we suppress the $\tilde{\cdot}$ and use the same notation $\cepm$ for both, the shifted function, and the function itself.
\end{notation}

We have the following existence and uniqueness result for the microscopic problem. A detailed proof can be found in \cite{GahnEffectiveTransmissionContinuous}, which can be easily extended to our slightly more general assumptions.

\begin{proposition}\label{ExistenceAprioriMicroscopicProblem}
There exists a unique weak solution $\cepm \in L^2((0,T),H^1(\Omega^{\pm}))$ with $\partial_t \cepm \in L^2((0,T),L^2(\Omega^{\pm}))$ and $\cem \in L^2((0,T),H^1(\oem))$ with $\partial_t \cem \in L^2((0,T),L^2(\oem))$ of the Problem $\eqref{MicroscopicModel}$ (satisfying the weak formulation $\eqref{VE_Omega}$). Additionally, the following a priori estimates are valid:
\begin{align*}
\|\partial_t \cepm \|_{L^2((0,T),L^2(\Omega^{\pm}))} + \|\cepm\|_{L^{\infty}((0,T),L^2(\Omega^{\pm}))} + \big\|\nabla \cepm \big\|_{L^2((0,T),L^2(\Omega^{\pm}))} &\le C,
\\
 \|\partial_t \cem\|_{L^2((0,T),L^2(\oem))} +  \|\cem\|_{L^{\infty}((0,T),L^2(\oem))} +   \big\|\nabla \cem \big\|_{L^2((0,T),L^2(\oem))} &\le C \sqrt{\epsilon}.
\end{align*}
\end{proposition}

\section{Main results}
\label{SectionMainResults}

In this section we state our main results. We consider approximations of $\ce$ including correctors of first and second order, leading to different orders of convergence with respect to the scaling parameter $\epsilon$. The definitions of the macroscopic solution of zeroth order and  the corrector terms can be found in Section \ref{SectionZerothOrderProblem} and  \ref{SectionAuxiliaryProblems}.

\subsection{Correctors including terms of order $\epsilon$ in the layer}
\label{SectionFirstOrderApproximation}

Let us define the first order approximation $c_{\epsilon,\app,1}:(0,T)\times \oe \rightarrow \R$ of $\ce$ by
\begin{align}
\begin{aligned}\label{DefFirstOrderApprox}
&c_{\epsilon,\app,1}(t,x):=  \begin{cases} c_{\epsilon,\app,1}^+(t,x) := c_0^+(t,x)   \, &\mbox{ for } x \in \Omega^+, 
\\
c_{\epsilon,\app,1}^M(t,x) := c_0^M(t,\x) + \epsilon c_1^M\left(t,\x,\fxe\right) \, &\mbox{ for } x \in \oem, 
\\
c_{\epsilon,\app,1}^- (t,x) :=  c_0^-(t,x)  \, &\mbox{ for } x \in \Omega^-. \end{cases}
\end{aligned}
\end{align}

\begin{remark}\label{RemarkUnstetigkeitApproximationErsterOrdnung}
The first order corrector term $c_1^M$ in the layer in the definition of $c_{\epsilon,\app,1}$ takes into account the oscillations within the thin layer. Adding this corrector leads to a discontinuity of $c_{\epsilon,\app,1}$ across the interfaces $\sepm$. Of course, it is also possible to add an additional first order corrector term $c_1^{\pm,\bl}$ in the bulk-domains $\oepm$ (see the definition of $c_{\epsilon,\app,2}$ below), however, this will not improve the order of convergence. 
\end{remark}

The error between the microscopic solution $\ce$ and the approximation $c_{\epsilon,\app,1}$ is estimated in the following theorem:

\begin{theorem}\label{MainTheoremFirstOrderApproximation}
Let $c_0^{\pm} \in L^2((0,T),H^2(\Omega^{\pm}))$ with $\partial_t c_0^{\pm} \in L^2((0,T),L^2(\Omega^{\pm}))$ and $\nabla_{\x} c_0^{\pm} \in L^{\infty}((0,T)\times \Omega^{\pm})$, and $c_0^M \in L^2((0,T),H^2(\Sigma))$ with $\partial_t c_0^M \in L^2((0,T),L^2(\Sigma))$. Then, the following error estimate is valid 
\begin{align*}
\sum_{\pm}\left\|\cepm - c_{\epsilon,\app,1}^{\pm}\right\|_{L^2((0,T),H^1(\Omega^{\pm}))} + \frac{1}{\sqrt{\epsilon}}\left\| \cem - c_{\epsilon,\app,1}^M\right\|_{L^2((0,T),H^1(\oem))} \le C_1 \sqrt{\epsilon},
\end{align*}
and the constant $C_1 = C_1(c_0^{\pm},c_0^M)>0$ fulfills
\begin{align*}
C_1 \le C \bigg( &1  + \big\| c_0^M\big\|_{L^2((0,T),H^2(\Sigma))} + \big\|\partial_t c_0^M\big\|_{L^2((0,T),L^2(\Sigma))} 
\\
&+ \sum_{\pm} \left[\big\|c_0^{\pm}\big\|_{L^2((0,T),H^2 (\Omega^{\pm}))} + \big\|\partial_t c_0^{\pm} \big\|_{L^2((0,T),L^2(\Omega^{\pm}))} + \big\|\nabla_{\x} c_0^{\pm}\big\|_{L^{\infty}((0,T)\times \Omega^{\pm})}\right]\bigg).
\end{align*}
\end{theorem}

Theorem \ref{MainTheoremFirstOrderApproximation} is a direct consequence of the more general result in Theorem \ref{GeneralErrorEstimateFirstOrderApproximation} in Section \ref{SectionErrorEstimateSecondOrder}.

\subsection{Correctors including terms up to order $\epsilon^2$ in the layer}
\label{SectionSecondOrderApproximation}

For higher order error estimates, we first have to overcome the problem of discontinuity across $\sepm$ of the approximation $c_{\epsilon,\app,1}$ by adding a first order corrector term in the bulk-domains, see Remark \ref{RemarkUnstetigkeitApproximationErsterOrdnung}. However, this leads to an additional normal flux from the bulk-regions into the thin layer. Therefore, we have to add an additional second order corrector (the diffusion in the layer is of order $\epsilon^{-1}$) in the layer. This leads to the following second order approximation:
We define $c_{\epsilon,\app,2}:(0,T)\times \oe \rightarrow \R$ by
\begin{align}
\begin{aligned} \label{DefSecondOrderApprox}
&c_{\epsilon,\app,2}(t,x):=  \begin{cases} c_{\epsilon,\app,2}^+(t,x) \, \mbox{ for } x \in \Omega^+
\\
c_{\epsilon,\app,2}^M(t,x) \, \mbox{ for } x \in \oem 
\\
c_{\epsilon,\app,2}^- (t,x) \, \mbox{ for } x \in \Omega^- \end{cases}
\\
&:= \begin{cases}
c_0^+(t,x) + \epsilon  c_1^{+,\bl}\left(t, x,\fxe\right) &\mbox{ for } x \in \Omega^+,
\\
c_0^M(t,\x) + \epsilon c_1^M\left(t, \x,\fxe\right) + \epsilon^2 c_2^M\left(t,\x,\fxe\right) &\mbox{ for } x \in \oem,
\\
c_0^-(t,x) + \epsilon  c_1^{-,\bl}\left(t, x,\fxe\right)  &\mbox{ for } x \in \Omega^-.
\end{cases}
\end{aligned}
\end{align}


\begin{remark}\label{RemarkUnstetigkeitApproximationZweiterOrdnung}
Again, the approximation $c_{\epsilon,\app,2}$ is not continuous across the interfaces $\sepm$.
\end{remark}

\begin{theorem}\label{MainTheoremSecondOrderApproximation}
Let $c_0^{\pm} \in L^2((0,T),H^2(\Omega^{\pm}))$ with $\partial_t c_0^{\pm} \in L^2((0,T),H^1(\Omega^{\pm}))$, and $c_0^M \in L^2((0,T),H^2(\Sigma))$ with $\partial_t c_0^M \in L^2((0,T),H^1(\Sigma))$.  Then, the following error estimate is valid
\begin{align*}
\sum_{\pm}\left\|\cepm - c_{\epsilon,\app,2}^{\pm}\right\|_{L^2((0,T),H^1(\Omega^{\pm}))} + \frac{1}{\sqrt{\epsilon}}\left\| \cem - c_{\epsilon,\app,2}^M\right\|_{L^2((0,T),H^1(\oem))} \le C_2\epsilon,
\end{align*}
and the constant $C_2 = C_2(c_0^{\pm},c_0^M)>0$ fulfills
\begin{align*}
C_2 \le& C\bigg( 1 + \|c_0^M\|_{L^2((0,T),H^2(\Sigma))} + \|\partial_t c_0^M \|_{L^2((0,T),H^1(\Sigma))} + \|c^{0,M}\|_{H^1(\Sigma)} \\
& \hspace{2em} + \sum_{\pm} \left[ \|c_0^{\pm}\|_{L^2((0,T),H^2(\Omega^{\pm}))} + \|\partial_t c_0^{\pm} \|_{L^2((0,T),H^1(\Omega^{\pm}))} + \|c^{0,\pm}\|_{H^1(\Omega^{\pm})}\right] \bigg).
\end{align*}
\end{theorem}
We see that the second order approximation $c_{\epsilon,\app,2}$ leads to a better error estimate with respect to the scaling parameter $\epsilon$ than $c_{\epsilon,\app,1}$. However, from the numerical point of view, this has to be paid by solving the cell and boundary layer problems for the correctors $c_2^M$ and $c_1^{\pm,\bl}$, see Section \ref{SectionAuxiliaryProblems}.

\section{The zeroth order macroscopic model }
\label{SectionZerothOrderProblem}

In this section we formulate the macroscopic problem of zeroth order. This was derived rigorously for a similar model in \cite{GahnEffectiveTransmissionContinuous} using the method of two-scale convergence and the unfolding operator. Here, we also take into account oscillations in the reactive term $\fepm$ and consider periodic boundary conditions on the lateral boundary instead of a Neumann-zero boundary condition. However, the results from \cite{GahnEffectiveTransmissionContinuous} still hold for our situation.

The macroscopic solution is defined in the following way: Let the triple $(c_0^+,c_0^M,c_0^-)$ with 
\begin{align*}
c_0^{\pm} &\in L^2((0,T),\hop(\Omega^{\pm})) \cap H^1((0,T),L^2(\Omega^{\pm})),
\\
c_0^M &\in L^2((0,T), \hop (\Sigma)) \cap H^1((0,T),L^2(\Sigma)),
\end{align*}
be the unique weak solution of the following transmission problem:
\begin{align}
\begin{aligned}\label{MacroscopicProblemZeroOrder}
\partial_t c_0^{\pm} - \nabla \cdot \left(D^{\pm} \nabla c_0^{\pm} \right) &= \int_{Y^n}f^{\pm}(t,y,c_0^{\pm}) dy &\mbox{ in }& (0,T)\times \Omega^{\pm},
\\
c_0^+ &= c_0^- = c_0^M &\mbox{ on }& (0,T)\times \Sigma,
\\
[\![D^{\pm} \nabla c_0^{\pm} \cdot \nu ]\!] &= |Z|\partial_t c_0^M - |Z|\nabla_{\x} \cdot \left(D^{M,\ast}\nabla_{\x} c_0^M\right) 
\\
&\hspace{3em} - \int_Z g^M(t,y,c_0^M) dy &\mbox{ on }& (0,T)\times \Sigma,
\\
-D^{\pm } \nabla c_0^{\pm} \cdot \nu^{\pm} &= 0 &\mbox{ on }& (0,T)\times \Sigma \times \{\pm H\},
\\
c_0^{\pm}, \, c_0^M &\mbox{ are } \Sigma\mbox{-periodic},
\\
c_0^{\pm}(0) &= c^{0,\pm} &\mbox{ in }& \Omega^{\pm},
\\
c_0^M(0) &= c^{0,M} &\mbox{ in }& \Sigma,
\end{aligned}
\end{align}
with $[\![D^{\pm} \nabla c_0^{\pm} \cdot \nu ]\!]:= -\big(D^{+} \nabla c_0^+ \cdot \nu^+ + D^{-} \nabla c_0^- \cdot \nu^-\big)$.
Here, $\nu^{\pm}$ denotes the outer unit normal  on $\partial \Omega^{\pm}$, and the homogenized diffusion coefficient   $D^{M,\ast}$ is defined by
\begin{align*}
D^{M,\ast}_{kl} &:= \frac{1}{|Z|}\int_Z D^M(y) \big(\nabla w_{k,1}^M + e_k \big) \cdot \big(\nabla w_{l,1}^M + e_l \big) dy \quad \mbox{for } k,l= 1,\ldots,n-1, 
\end{align*}
 where $w_{k,1}^M$ are the solutions of the cell problems  $\eqref{CellProblemFirstOrderLayer}$ in Section \ref{SectionAuxiliaryProblems}. The variational formulation for Problem $\eqref{MacroscopicProblemZeroOrder}$ is the following one: For all $(\phi^+,\phi^M,\phi^-) \in H^1(\Omega^+)\times H^1(\Sigma)\times H^1(\Omega^-)$ with $\phi^{\pm}|_{\Sigma} = \phi^M$ it holds almost everywhere in $(0,T)$
\begin{align}
\begin{aligned}
\label{VE_MacrEquaZero}
\sum_{\pm} &\left[ \int_{\Omega^{\pm}} \partial_t c_0^{\pm} \phi^{\pm} dx + \int_{\Omega^{\pm}} D^{\pm} \nabla c_0^{\pm} \cdot \nabla \phi^{\pm} dx \right]
\\
&\hspace{4em}+ |Z|\int_{\Sigma} \partial_t c_0^M \phi^M d\x + |Z|\int_{\Sigma} D^{M,\ast} \nabla_{\x} c_0^M \cdot \nabla_{\x} \phi^M d\x 
\\
&= \int_{\Sigma}\int_Z g^M(t,y,c_0^M) \phi^M dy d\x +  \sum_{\pm} \int_{\Omega^{\pm}}\int_{Y^n} f^{\pm}(t,y,c_0^{\pm}) \phi^{\pm} dy dx.
\end{aligned}
\end{align}

In the following we will prove regularity results for the macroscopic solution $(c_0^+, c_0^M , c_0^-)$ under additional assumptions on the data. Hence, let the Assumptions \ref{VoraussetzungenDuenneSchichtDiffusionskoeffizienten} - \ref{VoraussetzungenDuenneSchichtAnfangsbedingungen} be valid and additionally it holds that
\begin{enumerate}
[label = (A\arabic*)']
\setcounter{enumi}{1}
\item \label{ZusatzbedingungFplusminus} The function $f^{\pm}$ is differentiable with respect to $t  $  with $\partial_t f^{\pm} \in L^{\infty}((0,T)\times Y^n \times \R)$.

\item \label{ZusatzbedingungGM}The function  $g^M$ is differentiable with respect to $t $ with  $\partial_t g^M \in L^{\infty}((0,T)\times Z \times \R)$. 

\item \label{ZusatzbedingungAW} There exists a constant $M_0\geq 0$, such that $\|\nabla_{\x} c^{0,\pm}\|_{L^{\infty}(\Omega^{\pm})} \le M_0$ (not for the $n$-th derivative) and $\|\nabla_{\x} c^{0,M}\|_{L^{\infty}(\Sigma)}\le M_0$.

\end{enumerate}

Then we obtain the following regularity result for the macroscopic solution, which is sufficient for the assumptions in Theorem \ref{MainTheoremFirstOrderApproximation} and \ref{MainTheoremSecondOrderApproximation} to hold:

\begin{proposition}\label{RegularityResultsMacroscopicSolution}
Let the conditions \ref{VoraussetzungenDuenneSchichtDiffusionskoeffizienten} - \ref{VoraussetzungenDuenneSchichtAnfangsbedingungen}  be fulfilled. Then the triple $(c_0^+ , c_0^M , c_0^- )$ has the following regularity property:
\begin{align*}
c_0^{\pm}  \in L^2((0,T),H^2(\Omega^{\pm})), \quad 
c_0^M  \in L^2((0,T),H^2(\Sigma)).
\end{align*}
If we additionally assume  \ref{ZusatzbedingungFplusminus} - \ref{ZusatzbedingungAW}, then we obtain
\begin{align*}
c_0^{\pm} &\in  H^1((0,T),H^1(\Omega)), \quad 
\nabla_{\x} c_0^{\pm} \in L^{\infty}((0,T)\times \Omega^{\pm}),\\
c_0^M &\in H^1((0,T),H^1(\Sigma)), \quad 
\nabla_{\x} c_0^M \in L^{\infty}((0,T)\times \Sigma).
\end{align*}
\end{proposition}
\begin{proof}
The proof can be found in Section \ref{SectionRegularity} in the Appendix.
\end{proof}

\begin{remark}
We emphasize that the proofs of the main results from Section \ref{SectionMainResults} hold under the Assumptions  \ref{VoraussetzungenDuenneSchichtDiffusionskoeffizienten} - \ref{VoraussetzungenDuenneSchichtAnfangsbedingungen}, and the conditions stated in Theorem \ref{MainTheoremFirstOrderApproximation} and \ref{MainTheoremSecondOrderApproximation}. The assumptions \ref{ZusatzbedingungFplusminus} - \ref{ZusatzbedingungAW}   give   sufficient conditions for which the assumptions from Theorem \ref{MainTheoremFirstOrderApproximation} and \ref{MainTheoremSecondOrderApproximation} are fulfilled, but are far away from being optimal.
\end{remark}

\section{Corrector terms}
\label{SectionAuxiliaryProblems}

%
%
%

In this section we define the corrector terms $c_1^{\pm,\bl}$, $c_1^M$, and $c_2^M$ used in the first and second order approximations $c_{\epsilon,\app,1}$ and $c_{\epsilon,\app,2}$ in Section \ref{SectionMainResults}, and investigate their regularity properties. The correctors are defined via the derivatives of the macroscopic solution $c_0$ multiplied by solutions of appropriate cell respectively boundary layer problems, see $\eqref{CellProblemFirstOrderLayer} $ - $ \eqref{CellProblemSecondOrderLayerM}$.
\vspace{1em}

\noindent\textbf{Cell problems of first order for the thin layer $\oem$: }

For $j=1,\ldots,n-1$ the function $w_{j,1}^M \in H_{\#}^1(Z)/\R$ solves the following cell problem:
\begin{align}
\begin{aligned}
\label{CellProblemFirstOrderLayer}
-\nabla_y \cdot \left( D^M \big[\nabla_y w_{j,1}^M + e_j\big] \right) &= 0 \quad \mbox{ in }  Z,
\\
-D^M \big[\nabla_y w_{j,1}^M + e_j \big] \cdot \nu &= 0 \quad \mbox{ on } S^{\pm},
\\
w_{j,1}^M \mbox{ is } Y^{n-1}\mbox{-periodic}, \, &\int_Z w_{j,1}^M dy = 0.
\end{aligned}
\end{align}

\begin{lemma}\label{CellProblemFirstOrderLayerExistence}
For every $j=1,\ldots,n-1$  there exists a unique solution $w_{j,1}^M \in W^{2,p}(Z)$ of Problem $\eqref{CellProblemFirstOrderLayer}$ for $p\in (1,\infty)$ arbitrary large. Especially, we have  
\begin{align*}
\|w_{j,1}^M\|_{C^1(\overline{Z})} \le C.
\end{align*}
\end{lemma}
\begin{proof}
This follows from the $L^p$-theory for the Neumann-problem for elliptic equations, see \cite[Chapter 2]{GrisvardEllipticProblems}. The inequality follows from the Sobolev embedding theorem for $p>n$.
\end{proof}
Now, we define the first order corrector $c_1^M$ in the thin layer via
\begin{align}\label{DefFirstOrderCorrectorLayer}
c_1^M(t,\x,y) &:= \sum_{j=1}^{n-1} \partial_{x_j} c_0^M(t,\x) w_{j,1}^M(y) &\mbox{ in }& (0,T)\times \Sigma \times Z.
\end{align}
By adding $\epsilon c_1^M\left(t,\x,\fxe\right)$  to the macroscopic solution $c_0^M$ in the thin layer, we take into account the oscillations in the layer and can prove error estimates for  the gradients in  $L^2$, see Theorem \ref{MainTheoremFirstOrderApproximation}.

\vspace{1em}
\noindent\textbf{Boundary layer corrector for the bulk-domains $\Omega^{\pm}$:}

Using the corrector $c_{\epsilon,\app,1}$, we obtain an error estimate of order $\epsilon^{\frac12}$, see Theorem \ref{MainTheoremFirstOrderApproximation}. To obtain a better error estimate, we add further corrector terms to the approximation $c_{\epsilon,\app,1}$. Firstly, we add the corrector $c_1^{\pm,\bl}$ to the macroscopic solution in the bulk domains, which eliminates the discontinuity across the interfaces $\sepm$ of the approximation $c_{\epsilon,\app,1}$.

Let us define the infinite stripes $Y^{\pm}$ and their interface $Y^0$ by 
\begin{align*}
Y^+ &:= Y^{n-1} \times (0,\infty), \\
Y^- &:= Y^{n-1} \times (-\infty,0), \\
Y^0 &:= Y^{n-1} \times \{0\}.
\end{align*}
For fixed $\omega >0$, we  define 
\begin{align*}
W_{\omega,\#}\big(Y^+\big):= \left\{ u \in \hop\big(Y^{n-1}\times (0,R)\big) \mbox{ for every } R>0; \, e^{\omega y_n} \nabla u \in L^2\big(Y^+\big)\right\},
\end{align*} 
and in the same way we define the space $W_{\omega,\#}\big(Y^-\big)$.
For $j=1,\ldots,n-1$, the function $w_{j,1}^{\pm,\bl} \in W_{\omega,\#}\big(Y^{\pm}\big)$ solves the following boundary layer problem
\begin{align}
\begin{aligned}
\label{BoundaryLayerBulk}
-\nabla_y \cdot \left(D^{\pm}  \nabla w_{j,1}^{\pm,\bl}\right)  &= 0 &\mbox{ in }& Y^{\pm},
\\
w_{j,1}^{\pm,\bl}(\y,0) &= w_{j,1}^M(\y,1) &\mbox{ on }& Y^0 ,
\\
w_{j,1}^{\pm,\bl} \mbox{ is } Y^{n-1}\mbox{-periodic},
\\
\nabla w_{j,1}^{\pm,\bl} \mbox{ decreases } &\mbox{exponentially for } y_n \to \pm \infty.
\end{aligned}
\end{align} 

\begin{lemma}\label{BoundaryLayerBulkExistence}
For every $j=1,\ldots,n-1$, there exists a unique solution $w_{j,1}^{\pm,bl} \in W_{\omega,\#}(Y^{\pm})$ for a  suitable $\omega>0$. Additionally, it holds that $w_{j,1}^{\pm,\bl} \in W^{2,p}_{\mathrm{loc}}(\R^n_{\pm}) \cap C^0\big(\overline{\R^n_{\pm}}\big)$ with $\R^n_{\pm} := \left\{ x \in \R^n\, : \, \pm x_n >0 \right\}$ for arbitrary large $p\in (1,\infty)$. Especially, we have for a constant $C>0$
\begin{align*}
\big\|w_{j,1}^{\pm,\bl}\big\|_{W^{1,\infty}(\R^n_{\pm})} \le C.
\end{align*}
\end{lemma}
\begin{proof}
The existence and uniqueness follows from \cite[Theorem 10.1]{LionsBookControl1981} and the regularity result from  Lemma \ref{CellProblemFirstOrderLayerExistence}. The local $W^{2,p}$-regularity follows from the $L^p$-theory for elliptic equations, which also implies the continuity of the solution. It remains to check that the solution and its gradient are bounded on $\R^n_{\pm}$. This follows from the weak maximum principle, see \cite[Theorem A.1.1]{NeussRaduDissertation1999} and also \cite[Section 8.1]{GilbargTrudingerEllipticEquations}. More precisely we have
\begin{align*}
\sup_{y \in \overline{Y^{\pm}}} \left|w_{j,1}^{\pm,\bl}(y)\right| \le \sup_{y \in \overline{Y^0}} \left|w_{j,1}^{\pm,\bl}(y)\right|.
\end{align*}
Now, the local estimate from \cite[Theorem 9.11]{GilbargTrudingerEllipticEquations} which holds uniformly on every ball with fixed radius in $\R^n$ and the Morrey inequality imply the boundedness of the gradient.
\end{proof}
We define the first order corrector term $c_1^{\pm,\bl}$  for the bulk-domains $\Omega^{\pm}$ by
\begin{align}\label{DefFirstOrderCorrectorBulk}
c_1^{\pm,\bl}(t,x,y) &:= \psi(x_n)\sum_{j=1}^{n-1} \partial_{x_j} c_0^{\pm}(t,x) w_{j,1}^{\pm,\bl}(y) &\mbox{ in }& (0,T)\times \Omega^{\pm} \times Y^{\pm}.
\end{align}
Here, we take  $\psi \in C_0^{\infty}(-H,H)$ with $0\le \psi \le 1$ and $\psi = 1$ in a neighborhood of $0$.
Now, the functions $c_{\epsilon,\app,2}^{\pm}= c_0^{\pm} + \epsilon c_1^{\pm,\bl}\left(\cdot,\frac{\cdot}{\epsilon}\right)$ and $c_{\epsilon,\app,1}^M = c_0^M + \epsilon c_1^M\left(\bar{\cdot},\frac{\cdot}{\epsilon}\right)$ coincide on the interfaces $\sepm$. 

\vspace{1em}
\noindent\textbf{Corrector of second order for the thin layer $\oem$:}

The corrector $\epsilon c_1^{\pm,\bl}$ leads to a  jump of the normal fluxes across $\sepm$. Therefore, we add an additional corrector of second order in the thin layer. We emphasize that due to the different scaling of the diffusion coefficients in the bulk-domains and the thin layer in the microscopic problem, we can expect correctors of different orders in the bulk domains and the layer.

For $j=1,\ldots,n-1$  the function $w_{j,2}^M \in \hop(Z)/\R$  solves the following cell problem:
\begin{align}
\begin{aligned}\label{CellProblemSecondOrderLayerM}
-\nabla_y \cdot \left(D^M \nabla_y w_{j,2}^M\right) &= 0 &\mbox{ in }& Z,
\\
-D^M\nabla_y w_{j,2}^M \cdot \nu &= -D^{\pm} \nabla_y w_{j,1}^{\pm,\bl}(\y,0)\cdot \nu &\mbox{ on }& S^{\pm},
\\
w_{j,2}^M \mbox{ is } &Y^{n-1}\mbox{-periodic, } \int_Z w_{j,2}^M dy = 0.
\end{aligned}
\end{align}

\begin{lemma}\label{CellProblemsSecondOrderLayerExistence}
For $j=1,\ldots,n-1$, there exists a  unique solution  $w_{j,2}^M  \in W^{2,p}(Z)$ of Problem $\eqref{CellProblemSecondOrderLayerM}$ for arbitrary large $p\in (1,\infty)$. Especially, we have
\begin{align*}
\|w_{j,2}^M\|_{C^1(\overline{Z})} \le C.
\end{align*}
\end{lemma}
\begin{proof}
Again, the claim follows from the $L^p$-theory for the elliptic equations with Neumann-boundary conditions, and the regularity results from Lemma \ref{BoundaryLayerBulkExistence}.
\end{proof}
We define the second order corrector $c_2^M$ in the thin layer by
\begin{align}\label{DefSecondOrderCorrectorLayer}
c_2^M(t,\x,y) &:=   \sum_{j=1}^{n-1} \partial_{x_j} c_0^M(t,\x) w_{j,2}^M(y) &\mbox{ in }& (0,T)\times \Sigma \times Z.
\end{align}

\section{Error estimates}
\label{SectionErrorEstimates}

In this section, we give the proof of our main results. Roughly speaking, the idea is to apply the microscopic differential operator from Problem  $\eqref{PDE_ReacDiffEquation}$ to the microscopic solution and the approximative solution $c_{\epsilon,\app,j}$ ($j=1,2$), and subtract these terms from each other. More precisely, we start from the following term:
For $\pepm \in C^{\infty}(\Omega^{\pm})$ and $\pem \in C^{\infty}(\oem)$ with $\pepm|_{\Sigma} = \pem|_{\sepm}$ we consider for $j=1,2$
\begin{align}
\begin{aligned} \label{StartingEquationErrorEstimates}
\sum_{\pm}& \left[ \int_{\Omega^{\pm}} \partial_t \big(\cepm - c_{\epsilon,\app,j}^{\pm}\big) \pepm dx + \int_{\Omega^{\pm}} D^{\pm}\nabla \big(\cepm - c_{\epsilon,\app,j}^{\pm} \big) \cdot \nabla \pepm dx  \right]
\\
&+ \foe \int_{\oem} \partial_t \big(\cem - c_{\epsilon,\app,j-1}^M \big) \pem dx + \foe \int_{\oem} D^M\left(\fxe\right) \nabla \big(\cem - c_{\epsilon,\app,j}^M \big)\cdot \nabla \pem dx 
\\
=:& \sum_{\pm} \left[A_{\epsilon}^{\pm,1} + A_{\epsilon}^{\pm,2}\right] + A_{\epsilon}^{M,3} + A_{\epsilon}^{M,4},
\end{aligned}
\end{align}
where  we use the short notation $c_{\epsilon,\app,0}^M:= c_0^M$. At this point, we do not give precise information about the regularity of the macroscopic solution $(c_0^+,c_0^M,c_0^-)$. This will be specified in the following results. 
Our aim is to estimate the terms in $\eqref{StartingEquationErrorEstimates}$ and to choose a suitable test function. However, as mentioned in Remark \ref{RemarkUnstetigkeitApproximationErsterOrdnung} and \ref{RemarkUnstetigkeitApproximationZweiterOrdnung}, the error function $c_{\epsilon} - c_{\epsilon,\app,j}$ is not an admissible test function. Therefore, we have to add an additional corrector term in the bulk-domains, what leads to an additional error.


We estimate the terms in $\eqref{StartingEquationErrorEstimates}$ for $j=2$, because the most error terms carry over to the case $j=1$.  First of all, by using $\eqref{DefFirstOrderCorrectorBulk}$ we obtain
\begin{align*}
A_{\epsilon}^{\pm,2} &= \int_{\Omega^{\pm}} D^{\pm}  \nabla \cepm \cdot \nabla \pepm  - D^{\pm} \nabla c_0^{\pm}\cdot\nabla \pepm 
 - \epsilon D^{\pm}\nabla \left(c_1^{\pm,\bl}\left(x,\fxe\right)\right) \cdot \nabla \pepm  dx
\\
&:= \sum_{j=1}^3 B_{\epsilon}^{\pm,j}.
\end{align*}
For the last term $B_{\epsilon}^{\pm,3}$, an elemental calculation gives
\begin{align*}
B_{\epsilon}^{\pm,3} &= -\epsilon \sum_{j=1}^{n-1} \int_{\Omega^{\pm}} \psi(x_n) w_{j,1}^{\pm,\bl}\left(\fxe\right) D^{\pm}  \nabla \partial_{x_j} c_0^{\pm} \cdot \nabla \pepm dx
\\
&\hspace{1em} - \epsilon \sum_{j=1}^{n-1} \int_{\Omega^{\pm}}  \psi^{\prime}(x_n) \partial_{x_j} c_0^{\pm} w_{j,1}^{\pm,\bl}\left(\fxe\right) D^{\pm}  e_n \cdot \nabla \pepm dx
\\
&\hspace{1em} - \sum_{i=1}^{n-1}\sum_{k = 1}^n \int_{\Omega^{\pm}} \psi(x_n)\left(\sum_{j=1}^n D_{jk}^{\pm} \partial_{y_j} w_{i,1}^{\pm,\bl}\left(\fxe\right)\right) \partial_{x_i} c_0^{\pm} \partial_{x_k} \pepm dx.
\end{align*}
Let us define the tensor $T^{\pm,\bl}: \overline{Y^{\pm}} \rightarrow \R^{n\times (n-1)}$ by
\begin{align}
\label{Tpmbl}
T^{\pm,\bl}_{ki}(y) &:= - \sum_{j=1}^n D_{jk}^{\pm}\partial_{y_j} w_{i,1}^{\pm,\bl}(y)
\end{align}
for $k=1,\ldots,n$ and $i=1,\ldots,n-1$.
This gives us  (we consider $\nabla_{\x}$ as both, a vector in $\R^{n-1}$ and in $\R^n$ with the last component equal to zero)
\begin{align*}
A_{\epsilon}^{\pm,2} =& \int_{\Omega^{\pm}} D^{\pm} \nabla \cepm\cdot \nabla \pepm dx - \int_{\Omega^{\pm}} D^{\pm} \nabla c_0^{\pm} \cdot \nabla \pepm dx 
\\
&+ \int_{\Omega^{\pm}} \psi(x_n) T^{\pm,\bl}\left(\fxe\right) \nabla_{\x} c_0^{\pm} \cdot \nabla \pepm dx
\\
& - \epsilon \sum_{j=1}^{n-1} \int_{\Omega^{\pm}} \psi(x_n)D^{\pm} \nabla \partial_{x_j} c_0^{\pm}  \cdot \nabla \pepm  \psi(x_n) w_{j,1}^{\pm,\bl}\left(\fxe\right) dx
\\
&-  \epsilon \sum_{j=1}^{n-1} \int_{\Omega^{\pm}} \psi^{\prime}(x_n) \partial_{x_j} c_0^{\pm} w_{j,1}^{\pm,\bl}\left(\fxe\right) D^{\pm} e_n \cdot \nabla \pepm dx.
\end{align*}
With similar arguments and by adding $D^{M,\ast}$ in a suitable way, we obtain for $A_{\epsilon}^{M,4}$ by using $\eqref{DefFirstOrderCorrectorLayer}$ and $\eqref{DefSecondOrderCorrectorLayer}$ (if it is necessary, we consider $D^{M,\ast}$ as an element of $\R^{n\times n}$ by setting the $n$-th row and column equal to zero):
\begin{align*}
A_{\epsilon}^{M,4}=& \foe \int_{\oem} D^M\left(\fxe\right) \nabla \cem \cdot \nabla \pem dx - \foe \int_{\oem} D^{M,\ast}\nabla_{\x} c_0^M \cdot \nabla \pem dx
\\
&+ \foe \int_{\oem} T^{M,1}\left(\fxe\right) \nabla_{\x}c_0^M \cdot \nabla \pem dx 
\\
&- \sum_{j=1}^{n-1} \int_{\oem} D^M\left(\fxe\right) \nabla_{\x} \partial_{x_j} c_0^M \cdot \nabla \pem w_{j,1}^M\left(\fxe\right) dx 
\\
&+ \int_{\oem} T^{M,2}\left(\fxe\right) \nabla_{\x} c_0^M\cdot \nabla \pem dx 
\\
&- \epsilon \sum_{j=1}^{n-1} \int_{\oem} D^M\left(\fxe\right) \nabla_{\x} \partial_{x_j} c_0^M \cdot \nabla \pem w_{j,2}^M\left(\fxe\right) dx,
\end{align*}
with
\begin{align}
\label{TMone}
T^{M,1} &: \overline{Z} \rightarrow \R^{n\times (n-1)},
\\
\label{TMtwo}
T^{M,2} &:\overline{Z} \rightarrow \R^{n\times (n-1)}
\end{align}
defined for $k\in \{1,\ldots,n\}$ and $ i \in \{1,\ldots,n-1\}$ by
\begin{align*}
T^{M,1}_{ki}(y) &:= D_{ik}^{M,\ast} - D_{ik}^{M}(y) -\sum_{j=1}^n D_{jk}^M(y) \partial_{y_j} w_{i,1}^M(y) ,
\\
T^{M,2}_{ki}(y) &:= -\sum_{j=1}^n D_{jk}^M(y) \partial_{y_j} w_{i,2}^M(y).
\end{align*}
Now, let us define the averaged function $\bpem \in H^1(\Sigma)$ by
\begin{align*}
\bpem(\x):= \frac{1}{2\epsilon}\int_{-\epsilon}^{\epsilon} \pem(\x,x_n) dx_n.
\end{align*}
We observe that with this definition we can write
\begin{align*}
\foe \int_{\oem} \partial_t &c_0^M \pem dx + \foe \int_{\oem} D^{M,\ast}  \nabla_{\x} c_0^M \cdot \nabla \pem dx 
\\
&= |Z|\int_{\Sigma} \partial_t c_0^M \bpem d\x + |Z|\int_{\Sigma} D^{M,\ast} \nabla_{\x} c_0^M \cdot \nabla_{\x} \bpem d\x.
\end{align*}
Altogether, we obtain for the term $\eqref{StartingEquationErrorEstimates}$:
\begin{align*}
\sum_{\pm}&\left[A_{\epsilon}^{\pm,1} + A_{\epsilon}^{\pm,2} \right] + A_{\epsilon}^{M,3} + A_{\epsilon}^{M,4}
\\
=& \sum_{\pm}\left[ \int_{\Omega^{\pm}} \partial_t \cepm \pepm dx + \int_{\Omega^{\pm}} D^{\pm}\nabla \cepm \cdot \nabla \pepm dx \right] 
\\
&\hspace{4em}+ \foe \int_{\oem} \partial_t \cem \pem dx + \foe \int_{\oem} D^M\left(\fxe\right)\nabla \cem \cdot \nabla  \pem dx
\\
&- \sum_{\pm} \left[ \int_{\Omega^{\pm}} \partial_t c_0^{\pm} \pepm dx + \int_{\Omega^{\pm}} D^{\pm} \nabla c_0^{\pm} \cdot \nabla \pepm dx \right] 
\\
&\hspace{4em}- |Z|\int_{\Sigma} \partial_t c_0^M \bpem d\x - |Z|\int_{\Sigma} D^{M,\ast} \nabla_{\x}c_0^M\cdot \nabla_{\x} \bpem d\x
\\
& + \Delta_{\epsilon,\partial_t} +\Delta_{\epsilon,T} + \Delta_{\epsilon,rest},
\end{align*}
with
\begin{align}
\Delta_{\epsilon,\partial_t}:= -\foe \int_{\oem} \epsilon \partial_t c_1^M\left(\x,\fxe\right) \pem dx -  \sum_{\pm} \epsilon \int_{\Omega^{\pm}} \partial_t c_1^{\pm,\bl}\left(x,\fxe\right) \pepm dx,  
\notag
\\
\begin{aligned}
\label{DeltaEpsT}
\Delta_{\epsilon,T}:=& \sum_{\pm} \left[ \int_{\Omega^{\pm}} \psi(x_n) T^{\pm,\bl}\left(\fxe\right) \nabla c_0^{\pm} \cdot \nabla \pepm dx\right]
\\
&+ \foe \int_{\oem} T^{M,1}\left(\fxe\right) \nabla_{\x} c_0^M\cdot \nabla \pem dx + \int_{\oem} T^{M,2}\left(\fxe\right) \nabla_{\x} c_0^M \cdot \nabla \pem dx,
\end{aligned}
\end{align}
and
\begin{align*}
\Delta_{\epsilon,rest} :=& -\epsilon \sum_{\pm}\sum_{j=1}^{n-1} \int_{\Omega^{\pm}} D^{\pm} \nabla \partial_{x_j} c_0^{\pm} \cdot \nabla \pepm  \psi(x_n) w_{j,1}^{\pm,\bl}\left(\fxe\right) dx 
\\
&- \epsilon \sum_{\pm} \sum_{j=1}^{n-1} \int_{\Omega^{\pm}} \psi^{\prime}(x_n) \partial_{x_j} c_0^{\pm} w_{j,1}^{\pm,\bl}\left(\fxe\right) D^{\pm}  e_n \cdot \nabla \pepm dx
\\
&- \sum_{j=1}^{n-1} \int_{\oem} D^M\left(\fxe\right) \nabla_{\x} \partial_{x_j} c_0^M \cdot \nabla \pem w_{j,1}^M\left(\fxe\right) dx 
\\
&- \epsilon \sum_{j=1}^{n-1} \int_{\oem} D^M\left(\fxe\right) \nabla_{\x} \partial_{x_j} c_0^M\cdot \nabla \pem  w_{j,2}^M \left(\fxe\right) dx .
\end{align*}
We emphasize that $(\pepm,\pem)$ is an admissible test function for the variational equation $\eqref{VE_Omega}$ of the microscopic problem, but $(\pepm,\bpem)$ is not an admissible test function for the variational equation $\eqref{VE_MacrEquaZero}$ of the macroscopic model of zeroth order. Therefore, in the following we use equation $\eqref{MacroscopicProblemZeroOrder}$ tested with $\big(\pepm, \bpem\big)$.
Due to the regularity of $c_0^{\pm}$ from Proposition \ref{RegularityResultsMacroscopicSolution}, the normal fluxes are elements of $L^2(\Sigma)$. Hence, from $\eqref{VE_Omega}$ and $\eqref{MacroscopicProblemZeroOrder}$ we obtain
\begin{align}
\begin{aligned}
\label{AuxiliaryEquation}
\sum_{\pm}&\left[A_{\epsilon}^{\pm,1} + A_{\epsilon}^{\pm,2} \right] + A_{\epsilon}^{M,3} + A_{\epsilon}^{M,4}
\\
=&\sum_{\pm} \left[ \int_{\Omega^{\pm}} \fepm(\cepm) \pepm dx - \int_{\Omega^{\pm}} \int_{Y^n} f^{\pm}(t,y,c_0^{\pm}) \pepm dy dx  \right]
\\
&+ \foe \int_{\oem} \gem(\cem)\pem dx - \int_{\Sigma} \int_Z g^M(t,y,c_0^M) \bpem dy d\x
\\
&+\sum_{\pm} \int_{\Sigma} D^{\pm} \nabla c_0^{\pm}\cdot \nu^{\pm} \big( \pepm - \bpem \big) d\sigma + \Delta_{\epsilon,\partial_t} + \Delta_{\epsilon,T} + \Delta_{\epsilon,rest}.
\end{aligned}
\end{align}
We have to estimate the terms on the right-hand side, where the most challenging term is $\Delta_{\epsilon,T}$. We start with $\Delta_{\epsilon,\partial_t}$, which only occurs for $j=2$.

\begin{lemma}[Estimate for $\Delta_{\epsilon,\partial_t}$]
\label{EstimateDeltaEpsPartialT}
For $c_0^{\pm} \in H^1((0,T),H^1(\Omega)^{\pm}))$ and $c_0^M\in H^1((0,T),H^1(\Sigma))$, for almost every $t \in (0,T)$ it holds that
\begin{align*}
\Delta_{\epsilon,\partial_t} \le C \left(\sqrt{\epsilon} \|\partial_t c_0^M\|_{H^1(\Sigma)} \|\pem\|_{L^2(\oem)} + \epsilon \|\partial_t c_0^{\pm}\|_{H^1(\Omega^{\pm})}\|\pepm\|_{L^2(\Omega^{\pm})} \right).
\end{align*}
\end{lemma}
\begin{proof}
The estimate follows easily from the H\"older-inequality and the essential boundedness of $w_{j,1}^{\pm,\bl}$ and $w_{j,1}^M$, see Lemma \ref{CellProblemFirstOrderLayerExistence} and \ref{BoundaryLayerBulkExistence}.
\end{proof}

\begin{remark}
Here we use the additional regularity for the time derivative of the macroscopic solution from the hypothesis in Theorem \ref{MainTheoremSecondOrderApproximation}. 
\end{remark}

\begin{lemma}[Estimate for $\Delta_{\epsilon,rest}$]
\label{EstimateDeltaEpsRest}
For $c_0^{\pm} \in L^2((0,T),H^2(\Omega^{\pm}))$ and $c_0^M \in L^2((0,T),H^2(\Sigma))$, for almost every $t \in (0,T)$ it holds that
\begin{align*}
\Delta_{\epsilon,rest} &\le C \left(\sqrt{\epsilon} \|c_0^M\|_{H^2(\Sigma)}\big\|\nabla \pem\big\|_{L^2(\oem)} + \epsilon \sum_{\pm} \|c_0^{\pm}\|_{H^2(\Omega^{\pm})} \big\|\nabla \pepm \big\|_{L^2(\Omega^{\pm})} \right).
%
\end{align*}
\end{lemma}
\begin{proof}
The claim follows easily from the continuity of $D^M$ and the regularity results from Lemma \ref{CellProblemFirstOrderLayerExistence}, \ref{BoundaryLayerBulkExistence} and \ref{CellProblemsSecondOrderLayerExistence}. 
\end{proof}

\begin{lemma}[Estimate for the interface term in $\eqref{AuxiliaryEquation}$]\label{ErrorEstimateNormalFlux}
Let $c_0^{\pm} \in L^2((0,T),H^2(\Omega^{\pm}))$. Then almost everywhere in $(0,T)$ it holds that
\begin{align*}
\int_{\Sigma} D^{\pm} \nabla c_0^{\pm} \cdot \nu^{\pm} \big(\pepm - \bpem \big) d\sigma \le C \sqrt{\epsilon} \|c_0^{\pm}\|_{H^2(\Omega^{\pm})} \big\|\nabla \pem \big\|_{L^2(\oem)}.
\end{align*}
\end{lemma}
\begin{proof}
The fundamental theorem of calculus and $\pepm|_{\Sigma} = \pem|_{\sepm}$ imply
\begin{align*}
\big\|\pepm|_{\Sigma} - \bpem \big\|_{L^2(\Sigma)} \le C \sqrt{\epsilon} \|\partial_n \pem \|_{L^2(\oem)} \le C \sqrt{\epsilon} \big\|\nabla \pem \big\|_{L^2(\oem)}.
\end{align*}
Now, the regularity of $c_0^{\pm}$ implies
\begin{align*}
\int_{\Sigma} D^{\pm} \nabla c_0^{\pm} \cdot \nu^{\pm} \big(\pepm - \bpem \big) d\sigma &\le C \big\|\nabla c_0^{\pm }\cdot \nu^{\pm } \big\|_{L^2(\Sigma)} \big\|\pepm|_{\Sigma} - \bpem \big\|_{L^2(\Sigma)} 
\\
&\le C \sqrt{\epsilon} \|c_0^{\pm}\|_{H^2(\Omega^{\pm})} \big\|\nabla \pem\big\|_{L^2(\oem)}.
\end{align*}
\end{proof}

\noindent\textbf{Estimates for the nonlinear terms in $\eqref{AuxiliaryEquation}$:}

Let us estimate now the differences including the nonlinear terms in $\eqref{AuxiliaryEquation}$. We start with the following auxiliary Lemma:
\begin{lemma}\label{AuxiliaryDivergenceEquation}
Let $c_0^M \in L^2((0,T), H^1(\Sigma))$ and let us define $\eta \in L^2((0,T)\times \Sigma,L^{\infty}(Z))$ by 
\begin{align*}
\eta(t,\x,y) := g^M\big(t,y,c_0^M(\x)\big) - \frac{1}{|Z|} \int_Z g^M\big(t,z,c_0^M(\x) \big) dz .
\end{align*}
There exists $G \in L^2((0,T),H^1_{\#}(\Sigma,W^{1,p}(Z)))$ for $p\in (1,\infty)$ arbitrary large, such that
\begin{align*}
\nabla_y \cdot G &= \eta  &\mbox{ in }& (0,T) \times \Sigma \times Z,
\\
G\cdot \nu &= 0 &\mbox{ on }& (0,T)\times \Sigma \times S^{\pm}.
\end{align*}
Especially, the following estimate holds:
\begin{align*}
\|G\|_{L^2\big((0,T),H^1\big(\Sigma,C^0\big(\overline{Z}\big)\big)\big) } \le C \left(1 + \|c_0^M\|_{L^2((0,T),H^1(\Sigma))}\right).
\end{align*}
\end{lemma}
\begin{proof}
Here, the time variable has the role of an additional parameter and for an easier notation we suppress the time-dependence in the following.

\textit{\underline{Step 1} (Existence of a solution $G\in H^1(\Sigma,W^{1,p}(Z))$): }First of all, the function $\eta$ is an element of $L^2(\Sigma,L^{\infty}(Z))$, since we obtain from the growth condition from Assumption \ref{VoraussetzungenDuenneSchichtReaktionskinetikMembran} 
\begin{align*}
\|\eta\|_{L^2(\Sigma,L^{\infty}(Z))} \le C \left(1 + \|c_0^M\|_{L^2(\Sigma)}\right).
\end{align*}
Further, the mean value with respect to $y$ is obviously zero. Additionally, the regularity of $g^M$ implies
\begin{align*}
\partial_{x_i} \eta(\x,y) = \partial_z g^M\big(y,c_0^M(\x)\big) \partial_{x_i} c_0^M(\x) - \frac{1}{|Z|} \int_Z \partial_z g^M(z,c_0^M(\x))\partial_{x_i} c_0^M(\x) dz
\end{align*}
for $i=1,\ldots,n-1$ and almost every $(\x,y)\in \Sigma \times Z$. Since the derivative of a Lipschitz-function is bounded (independent of $y$, due to our assumptions), we obtain
\begin{align*}
\big\|\partial_{x_i} \eta \big\|_{L^{\infty}(Z,L^2(\Sigma))} \le C \|c_0^M\|_{H^1(\Sigma)},
\end{align*}
\ie $\eta \in L^{\infty}(Z,H^1(\Sigma))$. Now, let $\xi \in L^2(\Sigma,H_{\#}^1(Z)/\R)$ be the  unique solution of the following problem:
\begin{align*}
- \Delta_y \xi(\x,y) &= \eta(\x,y) &\mbox{ in }& \Sigma \times Z,
\\
-\nabla_y \xi(\x,y)\cdot \nu &= 0 &\mbox{ on }& \Sigma \times S^{\pm},
\\
\xi \mbox{ is } Y^{n-1}\mbox{-periodic and } &\int_Z \xi dy = 0.
\end{align*}
The $L^p$-theory for elliptic equations implies $\xi \in L^2(\Sigma,W^{2,p}(Z))$ for $p\in (1,\infty)$ arbitrary, and
\begin{align*}
\|\xi\|_{L^2(\Sigma,W^{2,p}(Z))} \le C \|\eta\|_{L^2(\Sigma,L^p(Z))} \le C\left(1 + \|c_0^M\|_{L^2(\Sigma)}\right).
\end{align*}
To prove regularity of $\xi$ with respect to $\x$, we use the method of difference quotients, see \cite[Section 7.11]{GilbargTrudingerEllipticEquations}. We define for $\psi \in L^2(\Sigma \times Z)$, $i\in \{1,\ldots,n-1\}$, and $h > 0$ the difference quotient by
\begin{align*}
\partial_h^i \psi(\x,y):= \frac{\psi(\x + he_i , y) - \psi(\x,y)}{h} \quad \mbox{ for } (\x,y) \in \Sigma_h \times Z,
\end{align*}
with $\Sigma_h := \{\x \in \Sigma \, : \, \mathrm{dist}(\x,\partial \Sigma)>h\}$. Then,  $\partial_h^i \xi$ is a solution of 
\begin{align*}
-\Delta_y \partial_h^i \xi(\x,y) &= \partial_h^i \eta(\x,y) &\mbox{ in }& \Sigma_h \times Z,
\\
-\nabla_y \partial_h^i \xi(\x,y) \cdot \nu &= 0 &\mbox{ on }& \Sigma_h \times S^{\pm},
\\
\partial_h^i \xi \mbox{ is } Y^{n-1}\mbox{-periodic and } &\int_Z \partial_h^i \xi dy = 0.
\end{align*}
By the same arguments as above we obtain $\partial_h^i \xi \in L^2(\Sigma_h,W^{2,p}(Z))$ and
\begin{align*}
\|\partial_h^i \xi\|_{L^2(\Sigma_h,W^{2,p}(Z))} \le C \|\partial_h^i \eta \|_{L^2(\Sigma_h,L^p(Z))} \le C \|\partial_{x_i} \eta \|_{L^2(\Sigma, L^p(Z))} \le C \|c_0^M\|_{H^1(\Sigma)}.
\end{align*}
The results from \cite[Section 7.11]{GilbargTrudingerEllipticEquations} extended to Banach valued functions implies $\xi \in H^1(\Sigma,W^{2,p}(Z))$. Especially, we can replace $\partial_h^i \xi$ in the problem above by $\partial_{x_i} \xi$.

Now, we define $G(\x,y) := -\nabla_y \xi(\x,y)$. This gives us the desired result for $G$, where the estimate follows from the continuous embedding $W^{1,p}(Z) \hookrightarrow C^0\big(\overline{Z}\big)$.

\textit{\underline{Step 2} (Periodicity of $G$ with respect to $\x$)}: We define the space of $L^2$-functions on $Z$ with mean value zero by $L^2_0(Z):= \left\{u\in L^2(Z)\, : \, \int_Z u dy = 0 \right\}$. Further, let $\omega = \Sigma$ or $\omega = \partial \Sigma$, and we define the linear operator
\begin{align*}
L_{\omega}: L^2(\omega,L_0^2(Z)) \rightarrow L^2(\omega,H^1_{\#}(Z)/\R),\quad L_{\omega}(\theta) = \xi,
\end{align*}
where $\xi $ is the unique weak solution of 
\begin{align*}
- \Delta_y \xi(\x,y) &= \theta(\x,y) &\mbox{ in }& \omega \times Z,
\\
-\nabla_y \xi(\x,y)\cdot \nu &= 0 &\mbox{ on }&  \omega \times  S^{\pm},
\\
\xi \mbox{ is } Y^{n-1}\mbox{-periodic and } &\int_Z \xi dy = 0.
\end{align*}
Existence and uniqueness follows as in Step 1 from the Lax-Milgram Lemma, which also implies the continuity of the operator $L_{\omega}$. We consider the following vector-valued trace operators (see \cite[Theorem 6.13]{ArendtKreuter2018})
\begin{align*}
T_{L_0^2(Z)}&: H^1(\Sigma,L_0^2(Z)) \rightarrow L^2(\partial \Sigma,L_0^2(Z)),
\\
T_{H_{\#}(Z)/\R} &: H^1(\Sigma,H_{\#}^1(Z)/\R) \rightarrow L^2(\partial \Sigma,H_{\#}^1(Z)/\R). 
\end{align*}
The claim is proved, if we show
\begin{align}\label{Commutation}
T_{H_{\#}(Z)/\R} \circ L_{\Sigma} = L_{\partial \Sigma} \circ T_{L_0^2(Z)} \quad \mbox{on } H^1(\Sigma,L_0^2(Z)).
\end{align}
With similar arguments as in Step 1, we obtain the regularity result
\begin{align*}
L_{\omega} \big(C^0(\overline{\omega},L_0^2(Z))\big) \subset C^0(\overline{\omega},H_{\#}(Z)/\R),
\end{align*}
hence, the identity $\eqref{Commutation}$ holds on $C^0(\overline{\Sigma},L_0^2(Z))$. The density of $C^0(\overline{\Sigma},L_0^2(Z))$ and the continuity of $L_{\omega}$ and the trace operators imply the desired result.
\end{proof}

\begin{proposition}\label{EstimateNonlinearTermsLayer}
Let $c_0^M\in L^2((0,T),H^1(\Sigma))$. Then it holds that
\begin{align*}
\foe &\int_{\oem} \gem(\cem)\pem dx - \int_{\Sigma} \int_Z g^M(t,y,c_0^M) \bpem dy d\x
\\ 
&\le \frac{C}{\epsilon} \big\|\cem - c_0^M\big\|_{L^2(\oem)} \|\pem\|_{L^2(\oem)} + C \sqrt{\epsilon} \left(1  + \|c_0^M\|_{H^1(\Sigma)}\right) \|\pem\|_{H^1(\oem)}.
\end{align*}
\end{proposition}
\begin{proof}
We have
\begin{align*}
\foe &\int_{\oem} g^M\left(\fxe,\cem\right) \pem dx - \int_{\Sigma} \int_Z  g^M(y,c_0^M) \bpem dy d\x 
\\
=& \foe \int_{\oem} \left[ g^M\left(\fxe,\cem\right) - g^M\left(\fxe,c_0^M\right) \right] \pem dx 
\\
 &+\foe \int_{\oem} \left[ g^M\left(\fxe,c_0^M\right) - \frac{1}{|Z|} \int_Z g^M\big(y,c_0^M\big)dy \right] \pem dx =: G_{\epsilon}^1 + G_{\epsilon}^2.
\end{align*}
For the first term $G_{\epsilon}^1$ we use the uniform Lipschitz continuity of $g^M$ to obtain
\begin{align*}
G_{\epsilon}^1 \le& \frac{C}{\epsilon} \big\|\cem - c_0^M \big\|_{L^2(\oem)} \|\pem \|_{L^2(\oem)}.
\end{align*}
For the second term $G_{\epsilon}^2$ we use Lemma \ref{AuxiliaryDivergenceEquation} to obtain
\begin{align*}
G_{\epsilon}^2 =& \foe \int_{\oem} \eta\left(\x,\fxe\right) \pem dx  = \foe \int_{\oem} \nabla_y \cdot G\left(\x,\fxe\right) \pem dx
\\
=& \int_{\oem} \left[ -\big(\nabla_{\x} \cdot G\big)\left(\x,\fxe\right) + \nabla_x \cdot \left(G\left(\x,\fxe\right)\right)\right] \pem dx
\\
=& - \int_{\oem} \big(\nabla_{\x} \cdot G\big)\left(\x,\fxe\right) \pem dx - \int_{\oem} G\left(\x,\fxe\right) \cdot \nabla \pem   dx
\\
&+  \int_{\partial \oem} G\left(\x,\fxe\right) \cdot \nu \pem d\sigma =: F_{\epsilon}^1 + F_{\epsilon}^2 + F_{\epsilon}^3.
\end{align*}
The last term $F_{\epsilon}^3$ vanishes, due to the periodicity of $G$ and $G\cdot \nu = 0$ on $S^+ \cup S^-$. Further, the estimate from Lemma \ref{AuxiliaryDivergenceEquation} implies
\begin{align*}
F_{\epsilon}^2 &\le \left\|G\left(\x,\fxe\right)\right\|_{L^2(\oem)} \big\|\nabla \pem\big\|_{L^2(\oem)} \le C\sqrt{\epsilon} \|G\|_{L^2 (\Sigma, C^0 (\overline{Z} ) )} \big\|\nabla \pem \big\|_{L^2(\oem)}
\\
&\le C\sqrt{\epsilon} \big(1 + \big\|c_0^M\big\|_{H^1(\Sigma)}\big) \big\|\nabla \pem \big\|_{L^2(\oem)}.
\end{align*}
In a similar way, we can estimate $F_{\epsilon}^1$. 
\end{proof}

\begin{proposition}\label{EstimateNonlinearTermsBulk}
Let $c_0^{\pm} \in L^2((0,T),H^1(\Omega^{\pm}))$. Then it holds that
\begin{align*}
\int_{\Omega^{\pm}} &\fepm(\cepm) \pepm dx - \int_{\Omega^{\pm}} \int_{Y^n} f^{\pm} (y,c_0^{\pm}) \pepm dy dx
\\
&\le C  \big\|\cepm - c_0^{\pm}\big\|_{L^2(\Omega^{\pm})} \|\pepm\|_{L^2(\Omega^{\pm})} + C \epsilon \big(1 + \|c_0^{\pm}\|_{H^1(\Omega^{\pm})}\big) \|\pepm\|_{H^1( \Omega^{\pm})}.
\end{align*}
\end{proposition}
\begin{proof}
We argue in the same way as in the proof of Proposition \ref{EstimateNonlinearTermsLayer}. We have
\begin{align*}
\int_{\Omega^{\pm}} &\fepm(\cepm) \pepm dx - \int_{\Omega^{\pm}} \int_{Y^n} f^{\pm} (y,c_0^{\pm}) \pepm dy dx 
\\
&= \int_{\Omega^{\pm}} \left[f^{\pm}\left(\fxe,\cepm\right) - f^{\pm}\left(\fxe,c_0^{\pm}\right) \right] \pepm dx
\\
&\hspace{1em} + \int_{\Omega^{\pm}} \left[f^{\pm}\left(\fxe,c_0^M\right) - \int_{Y^n} f^{\pm}(y,c_0^{\pm})dy \right] \pepm dx
\\
&\le C \big\|\cepm - c_0^{\pm}\big\|_{L^2(\Omega^{\pm})} \|\pepm\|_{L^2(\Omega^{\pm})} + \int_{\Omega^{\pm}} \nabla_y \cdot F^{\pm}\left(x,\fxe\right) \pepm dx,
\end{align*}
where $F \in H^1(\Omega^{\pm},W^{1,p}(Z))$ for arbitrary large $p\in (1,\infty)$ fulfills (we again suppress the time-dependence)
\begin{align*}
\nabla_y \cdot F^{\pm}(x,y) = f^{\pm}(y,c_0^{\pm}) - \int_{Y^n} f^{\pm}(z,c_0^{\pm}) dz &\quad \mbox{in } \Omega^{\pm}\times Y^n,
\end{align*}
and the estimate
\begin{align*}
\|F^{\pm}\|_{H^1(\Omega^{\pm},C^0(\overline{Y^n}))} \le C \left(1 + \|c_0^{\pm}\|_{H^1(\Omega^{\pm})} \right).
\end{align*}
The existence and the regularity of $F^{\pm}$ can be established in the same way as in Lemma \ref{AuxiliaryDivergenceEquation}. However, we emphasize that we do not require specific boundary conditions for $F^{\pm}$ on $\partial \Omega^{\pm}$. We obtain
\begin{align*}
\int_{\Omega^{\pm}} \nabla_y \cdot F^{\pm } \left(x,\fxe\right) \pepm dx &= -\epsilon \int_{\Omega^{\pm}} (\nabla_x \cdot F^{\pm}) \left(x,\fxe\right) \pepm dx - \epsilon \int_{\Omega^{\pm}} F^{\pm}\left(x,\fxe\right) \cdot \nabla \pepm dx
\\
& + \epsilon \int_{\partial \Omega^{\pm}} F^{\pm}\left(x,\fxe\right) \cdot \nu \pepm d\sigma.
\end{align*}
We only consider the boundary term in more detail. The vector valued trace inequality implies
\begin{align*}
\epsilon \int_{\partial \Omega^{\pm}} F^{\pm}\left(x,\fxe\right) \cdot \nu \pepm d\sigma &\le \epsilon \|F^{\pm}\|_{L^2(\partial \Omega^{\pm},C^0(\overline{Y^n}))} \|\pepm\|_{L^2(\partial \Omega^{\pm})}
\\
&\le C\epsilon \big(1 + \|c_0^{\pm}\|_{H^1(\Omega^{\pm})}\big) \|\pepm\|_{H^1( \Omega^{\pm})}
\end{align*}
\end{proof}

\noindent\textbf{Estimate for $\Delta_{\epsilon,T}$}

Now, we estimate the term $\Delta_{\epsilon,T}$, where according to $\eqref{DeltaEpsT}$ we use the following notations for the included terms:
\begin{align*}
\Delta_{\epsilon,T}=:\sum_{\pm} \left[ \Delta_{\epsilon,T^{\pm,\bl}}\right] + \Delta_{\epsilon,T^{M,1}} + \Delta_{\epsilon,T^{M,2}}.
\end{align*}
The mean idea is to represent solenoidal vector fields by the divergence of skew-symmetric matrices and integrate by parts. This gives an additional factor $\epsilon$. This approach has been used in  \cite[Section 4.2]{JikovKozlovOleinik1994} for vector fields on $Y^n$, periodic in all directions, and \cite{neuss2001boundary} for boundary layers. In our case, we have to construct skew-symmetric matrices adapted to the structure of our problem. More precisely, these matrices have to be such that boundary terms which occur at the interfaces between the bulk domains and the thin layer vanish.

We start with the estimate for  the term $\Delta_{\epsilon,T^{M,1}}$. Therefore, we make use of  the following Lemma:

\begin{lemma}
\label{SkewSymmetricTensorOrderOneLayer}
Let $h=(h_1,\ldots,h_n) \in L^p(Z)^n$ for $1<p<\infty$ with 
\begin{align*}
\nabla_y \cdot h &= 0 &\mbox{ in }& Z,
\\
h\cdot \nu &= 0 &\mbox{ on }& S^+\cup S^-,
\\
h \mbox{ is } Y^{n-1}&\mbox{-periodic}, \, \int_Z h dy = 0.
\end{align*}
This means, that for all $\phi \in W^{1,p'}(Z)$ which are $Y^{n-1}$-periodic ($p'$ the dual exponent of $p$) it holds that
\begin{align*}
\int_Z h \cdot \nabla_y \phi dy = 0.
\end{align*}
Then, there exists a skew-symmetric tensor $\big(\beta_{il}\big) \in W^{1,p}(Z)^{n\times n}$ for $i,l=1,\ldots n$, such that
\begin{align*}
\sum_{l=1}^n \partial_{y_l} \beta_{il} &= h_i &\mbox{ in }& Z,
\\
\sum_{l=1}^n \nu_l \beta_{il} &= 0 &\mbox{ on }& S^+\cup S^-.
\end{align*} 
\end{lemma}
\begin{proof}
We define 
\begin{align*}
\beta_{il} := \partial_{y_l} \zeta_i - \partial_{y_i} \zeta_l,
\end{align*}
where $\zeta_i$ is the solution of 
\begin{align*}
\Delta \zeta_i &= h_i &\mbox{ in }& Z,
\\
\nabla \zeta_i \cdot \nu &= 0 &\mbox{ on }& S^+\cup S^- \mbox{ for } i=1,\ldots,n-1,
\\
\zeta_n &=0 &\mbox{ on }& S^+\cup S^-,
\\
\zeta_i \mbox{ is } Y^{n-1}&\mbox{-periodic for } i=1,\ldots,n,
\\
\int_Z \zeta_i dy &= 0 \mbox{ for } i=1,\ldots,n-1.
\end{align*}
From the $L^p$-theory for elliptic equations we obtain $\zeta_i \in W^{2,p}(Z)$ and therefore $\beta_{il} \in W^{1,p}(Z)$. Obviously, the boundary conditions for $\zeta_i$ on $S^+\cup S^-$ imply $\sum_{l=1}^n \nu_l \beta_{il} = \beta_{in}= 0$. Further, we have
\begin{align*}
\sum_{l=1}^n \partial_{y_l} \beta_{il} = \Delta \zeta_i - \sum_{l=1}^n \partial_{y_i y_l} \zeta_l = h_i - \partial_{y_i} \nabla_y \cdot \zeta,
\end{align*}
where we defined $\zeta:= (\zeta_1,\ldots,\zeta_n)$. We show that $v:= \nabla_y \cdot \zeta  =0 $. First of all, we have $v$ is $Y^{n-1}$-periodic and 
\begin{align*}
\int_Z v dy  = \int_{\partial Z} \zeta \cdot \nu d\sigma = \int_{S^+ \cup S^- } \zeta_n \nu_n d\sigma = 0, 
\end{align*}
due to the zero-boundary condition of $\zeta_n$ on $S^+\cup S^-$. Further, for all $Y^{n-1}$-periodic $\phi \in C^{\infty}\big(\overline{Z}\big)$ we have
\begin{align*}
\int_Z \nabla v \cdot \nabla \phi dy &= \sum_{i=1}^n \int_Z \partial_{y_i} \big(\nabla \zeta_i \cdot \nabla \phi \big) - \nabla \zeta_i \cdot \nabla (\partial_{y_i} \phi ) dy 
\\
&= \sum_{i=1}^n \int_{\partial Z} \nu_i \nabla \zeta_i \cdot \nabla \phi d\sigma - \sum_{i=1}^n \int_Z \nabla \zeta_i \cdot \nabla (\partial_{y_i} \phi ) dy
\\
&= \sum_{i,j=1}^n \int_{\partial Z} \nu_i \partial_{y_j} \zeta_i \partial_{y_j} \phi d\sigma  + \sum_{i=1}^n \int_Z \Delta \zeta_i \partial_{y_i} \phi dy  - \sum_{i,j=1}^n \int_{\partial Z} \nu_j \partial_{y_j} \zeta_i \partial_{y_i} \phi d\sigma
\\
&=: I_1 + I_2 - I_3.
\end{align*} 
For the second term we obtain from the properties of $h$ and the definition of $\zeta_i$
\begin{align*}
I_2= \int_Z h \cdot \nabla \phi dy = 0.
\end{align*}
Further, the periodicity and the Dirichlet-zero boundary condition of $\zeta_n$ imply 
\begin{align*}
I_1 &= \sum_{i,j=1}^n \int_{S^+ \cup S^-} \nu_i \partial_{y_j} \zeta_i \partial_{y_j} \phi d\sigma = \sum_{j=1}^n \int_{S^+ \cup S^-} \nu_n \partial_{y_j} \zeta_n  \partial_{y_j} \phi d\sigma 
\\
&= \int_{S^+\cup S^- } \nu_n \partial_{y_n} \zeta_n \partial_{y_n} \phi d\sigma.
\end{align*}
Using the Neumann-boundary conditions for $\zeta_i$ for $i=1,\ldots,n-1$, we get by similar arguments
\begin{align*}
I_3 = \sum_{i=1}^n \int_{S^+ \cup S^-} \nu_n \partial_{y_n} \zeta_i \partial_{y_i} \phi d\sigma = \int_{S^+ \cup S^-} \nu_n \partial_{y_n} \zeta_n \partial_{y_n} \phi d\sigma = I_1.
\end{align*}
Altogether, we obtain
\begin{align*}
\int_Z \nabla v \cdot \nabla \phi d y = 0,
\end{align*}
and therefore $v$ satisfies 
\begin{align*}
-\Delta v &= 0 &\mbox{ in }& Z,
\\
-\nabla v \cdot \nu &= 0 &\mbox{ on }& S^+ \cup S^-,
\\
v \mbox{ is } Y^{n-1}\mbox{-periodic and }& \int_Z v dy = 0.
\end{align*}
This implies $v= 0$ and the proof is complete.
\end{proof}

\begin{lemma}\label{ErrorEstimateDeltaTMOne}
Let $c_0^M\in L^2((0,T),H^2(\Sigma))$. Then it holds that
\begin{align*}
\Delta_{\epsilon,T^{M,1}} \le C\sqrt{\epsilon}\|c_0^M\|_{H^2(\Sigma)} \|\nabla \pem\|_{L^2(\oem)}.
\end{align*}
\end{lemma}
\begin{proof}
We have
\begin{align*}
\Delta_{\epsilon,T^{M,1}}= \foe \sum_{i=1}^{n-1} \sum_{j=1}^n \int_{\oem} T^{M,1}_{ji}\left(\fxe\right)\partial_{x_i} c_0^M \partial_{x_j} \pem dx.
\end{align*}
Let us define for fixed $i \in \{1,\ldots,n-1\}$ the function $h=(h_1,\ldots,h_n)$ by $h_j(y) := T_{ji}^{M,1}(y)$. An elemental calculation shows that $h$ fulfills the assumptions of Lemma \ref{SkewSymmetricTensorOrderOneLayer}. Hence, there exists for every $i\in \{1,\ldots,n-1\}$ a skew symmetric tensor $\big(\beta_{jl}^i\big) \in W^{1,p}(Z)$ (for $p\in (1,\infty)$ arbitrary large), such that $\sum_{l=1}^n \partial_{y_l} \beta_{jl}^i = T_{ji}^{M,1}$ in $Z$ and $\beta_{jn}^i = 0$ on $S^+ \cup S^-$. This implies by integration by parts
\begin{align*}
\Delta_{\epsilon,T^{M,1}} &= \foe \sum_{i=1}^{n-1} \sum_{j,l=1}^n \int_{\oem} \partial_{y_l} \beta_{jl}^i\left(\fxe\right) \partial_{x_i} c_0^M \partial_{x_j} \pem dx
\\
&= -\sum_{i=1}^{n-1} \sum_{j,l=1}^n \int_{\oem}  \beta_{jl}^i \left(\fxe\right) \left[ \partial_{x_l x_i} c_0^M \partial_{x_j}\pem + \partial_{x_i} c_0^M \partial_{x_l x_j}\pem \right] dx
\\
& \hspace{2em} + \sum_{i=1}^{n-1} \sum_{j,l=1}^n \int_{\partial \oem} \beta_{jl}^i\left(\fxe\right) \nu_l \partial_{x_i} c_0^M \partial_{x_j} \pem d\sigma.
\end{align*}
The boundary term vanishes due to the boundary conditions of $c_0^M$  and $\pem$ on the lateral boundary, and the boundary condition of $\beta_{jn}^i$ on $S^+ \cup S^-$. The terms including the derivatives $\partial_{x_l x_j} \pem$ vanish due to the skew-symmetry of $\beta_{jl}^i$ and the symmetry of the Hesse-matrix of $\pem$. Then, the regularity of $\beta_{jl}^i$ implies
\begin{align*}
\Delta_{\epsilon,T^{M,1}} = -\sum_{i=1}^{n-1} \sum_{j,l=1}^n \int_{\oem}  \beta_{jl}^i \left(\fxe\right) \partial_{x_l x_i} c_0^M \partial_{x_j}\pem dx \le C\sqrt{\epsilon}\|c_0^M\|_{H^2(\Sigma)} \|\nabla \pem\|_{L^2(\oem)}.
\end{align*}
\end{proof}

It remains to estimate the term $\sum_{\pm} \Delta_{\epsilon,T^{\pm,\bl}} + \Delta_{\epsilon,T^{M,2}}$. 
First of all we define for $-\infty <a < b < \infty$
\begin{align*}
\omega(a,b):=  (0,1)^{n-1} \times (a,b),
\end{align*}
and the space
\begin{align*}
V:= \left\{u \in L_{\loc}^2(\overline{\zi}) \, : \, \nabla u \in L^2(Z_{\infty}) ,\, u \mbox{ is } Y^{n-1}\mbox{-periodic}\right\},
\end{align*}
where $L_{\loc}^2(\overline{\zi})$ denotes the space of functions belonging to $L^2(U)$ for every $U\subset \zi$ such that $\overline{U}$ is compact. In a similar way we define other  Sobolev spaces which are integrable locally on $\zi$. For $u\in V$ we define the mean value of $u$ over $Y^{n-1} \times \{s\}$ for $s \in \R$ by
\begin{align*}
\u (s):= \int_{Y^{n-1}} u(\y,s) d\y.
\end{align*}
On $V$ we have the following weighted Poincar\'e-type inequality (see \cite[Prop. 1.3]{JaegerMikelic1996})
\begin{align}\label{PoincareWeighted}
\left\|\frac{1}{1 + |y_n|} \big( u - \u(0)\big)\right\|_{L^2(\zi)} \le C \|\nabla u \|_{L^2(\zi)}.
\end{align}
Therefore, the space
\begin{align*}
V_0:= \left\{u \in V \, : \, \u(0) = 0\right\}
\end{align*}
becomes a Hilbert space with respect to the inner product 
\begin{align*}
(u,v)_{V_0} := \int_{\zi} \nabla u \cdot \nabla v dy.
\end{align*}

\begin{lemma}\label{ExistenceLaplaceWeightedSpace}
Let $(1 +|y_n|) h_i \in L^2(\zi) $ with $\int_{\zi} h_i dy = 0$ for $i=1,\ldots,n$. Then there exists a unique weak solution $u \in V_0$ of the problem
\begin{align}
\label{LaplaceInfiniteStripe}
\begin{aligned}
-\Delta u &= h_i \mbox{ in } \zi,
\\
\u(0) &= 0 \mbox{ and } u \mbox{ is } Y^{n-1}\mbox{-periodic}.
\end{aligned}
\end{align}
\end{lemma}
\begin{proof}
The proof is based on the inequality $\eqref{PoincareWeighted}$ and can be found in \cite[Prop. 1.5]{JaegerMikelic1996}.
\end{proof}
Next we show additional regularity results for the solution $u$ from Lemma \ref{ExistenceLaplaceWeightedSpace} under additional assumptions on $h_i$. We use similar methods as in \cite{JaegerMikelic1996,Oleinik1992}. However, for the sake of completeness we give the detailed proof for our case.

\begin{lemma}\label{RegularityLaplaceWeighted}
For $\gamma >0$ let $e^{\gamma | y_n|} h_i \in L^2(\zi)$, $\int_{\zi } h_i dy = 0$, and $h_i\in L_{\loc}^p(\overline{\zi})$ for $p>n$ such that for every $s \in \R$ it holds that
\begin{align}\label{LocalConditionhi}
\|h_i\|_{L^p(\omega(s,s+1))} \le C,
\end{align}
with a constant $C >0$ independent of $s$.
 Then the solution $u$ from Lemma \ref{ExistenceLaplaceWeightedSpace} fulfills $u\in L_{\loc}^2(\overline{\zi})$ with $\|u\|_{L^2(\omega(s , s+ 1))} \le C$ uniformly with respect to $s$. Especially we have $u\in W^{2,p}_{\loc}(\overline{\zi})$ and $u \in C^1(\overline{\zi})$.
\end{lemma}
\begin{proof}
The $L^p$-theory of elliptic operators implies  $u \in W_{\loc}^{2,p}(\overline{\zi})$. Now, we prove ($Y^+ = Y^{n-1} \times (0,\infty)$)
\begin{align*}
\partial_n \u (0) = \int_{Y^+} h_idy .
\end{align*}
In fact, let $s>0$ and $\phi_s \in C^{\infty}(\R)$ with $0 \le \phi_s \le 1$, $\phi_s = 0 $ in $[s +1,\infty)$, $\phi = 1 $ in $(-\infty,s)$, and $\|\phi'\|_{\infty}\le 2$. Then by testing $\eqref{LaplaceInfiniteStripe}$ with $\phi_s$ and integrating over $Y^+$, we obtain
\begin{align*}
-\partial_n \u(0) = \int_{\omega(s,s+1)} \partial_n u \phi_s'dy - \int_{Y^+}h_i\phi_s dy \overset{s\to \infty}{\longrightarrow} - \int_{Y^+} h_i dy,
\end{align*}
since $\nabla u \in L^2(\zi)$. Using again the equation $-\Delta u = h_i$, we get for $\xi > 0$ 
\begin{align*}
\partial_n \u(\xi) = -\partial_n \u(0) + \int_{\omega(0,\xi)} h_i dy = - \int_{\omega(\xi,\infty)} h_idy.
\end{align*}
Integration with respect to $\xi$ from $0 $ to $s>0$ implies ($\u(0) = 0$)
\begin{align*}
|\u(s)| &= \left| \int_0^{s} \int_{\omega(\xi,\infty)} h_i dy d\xi \right| \le \int_0^{s} \big\|e^{\gamma \xi} h_i \big\|_{L^2(\omega(\xi,\infty))} \big\|e^{-\gamma \xi} \big\|_{L^2(\omega(\xi,\infty))} d\xi
\\
&\le C \big\|e^{\gamma y_n} h_i \big\|_{L^2(\zi)}
\end{align*}
for a constant $C>0$ independent of $s$. Now, the Poincar\'e-inequality implies for every $s\geq 0$
\begin{align*}
\|u\|_{L^2(\omega(s,s+1))} &\le \|u - \u(s)\|_{L^2(\omega(s,s+1))} + |\u(s)|
\\
&\le C \left( \|\nabla u \|_{L^2(\zi)} + \big\|e^{\gamma y_n} h_i\big\|_{L^2(\zi)}\right).
\end{align*}
The same arguments hold for $s<0$, what implies $u\in L_{\loc}^2(\overline{\zi})$. The claim follows from Theorem 8.17, 9.11, and 8.32 from \cite{GilbargTrudingerEllipticEquations}.
\end{proof}

Now we are able to construct the skew-symmetric tensors corresponding to the error terms $\Delta_{\epsilon,T^{\pm,\bl}}$ and $\Delta_{\epsilon,T^{M,2}}$.

\begin{lemma}\label{SkewSymmetricTensorBoundaryLayer}
For $\gamma > 0$ let $e^{\gamma y_n}h\in L^2(\zi)^n$ and $h \in L^p_{\loc}(\overline{\zi})^n$ for $p>n$ such that $\eqref{LocalConditionhi}$ holds for all $s \in \R$, and
\begin{align*}
\nabla \cdot h &= 0 \mbox{ in } \zi, 
\\
h \mbox{ is } Y^{n-1}&\mbox{-periodic and } \int_{\zi} h dy   = 0.
\end{align*}
More precisely, the conditions $\nabla \cdot h = 0$ and $h$ is $Y^{n-1}$-periodic means that for all $\phi \in C_{0,\#}^{\infty}(\zi)$ it holds that 
\begin{align*}
\int_{\zi} h \cdot \nabla \phi dy = 0.
\end{align*}
Then there exists a $Y^{n-1}$-periodic, skew symmetric tensor $\big(\beta_{kl}\big) \in L^2(\zi) \cap W^{1,p}_{\loc}(\overline{\zi})\cap C^0(\overline{\zi})$, such that
\begin{align}\label{DivergenceEquationBeta}
\sum_{l=1}^n \partial_l \beta_{kl} = h_k \quad \mbox{in } \zi.
\end{align}
\end{lemma}
\begin{proof}
From Lemma \ref{ExistenceLaplaceWeightedSpace} we obtain the existence of a function $\zeta_k $ such that $\Delta \zeta_k = h_k $ in $\zi$. Now, we define $\beta_{kl}:= \partial_l \zeta_k - \partial_k \zeta_l$ for $l,k=1,\ldots,n$. The regularity of $\beta_{kl}$ follows immediately from Lemma \ref{RegularityLaplaceWeighted} and we only have to check the   claim that $\eqref{DivergenceEquationBeta}$ holds. The definition of $\beta_{kl}$ implies (as in the proof of Lemma \ref{SkewSymmetricTensorOrderOneLayer})
\begin{align*}
\sum_{l=1}^n \partial_l \beta_{kl} = \Delta \zeta_k - \partial_k (\nabla \cdot \zeta) = h_k - \partial_k v,
\end{align*}
with $\zeta := (\zeta_1,\ldots,\zeta_n)$ and $v:= \nabla \cdot \zeta$.  Then, $\eqref{DivergenceEquationBeta}$ follows if we show $\nabla v = 0 $.

First of all, for $R>0$ we define $Z_R:= Y^{n-1} \times (-R,R)$ and $\psi \in C_0^{\infty}(-R,R)$ with $0 \le \psi \le 1$, $\psi = 1 $ in $[-R + 1 ,R -1]$, and $||\psi'||_{\infty}\le 2$. For every $\phi \in C_{0,\#}^{\infty}(\zi)$ we obtain
\begin{align*}
\int_{Z_R} \psi(y_n) \nabla v \cdot \nabla \phi dy &=  \sum_{i=1}^n \int_{Z_R} \psi(y_n) \left( \partial_i \big(\nabla \zeta_i \cdot \nabla \phi  \big) - \nabla \zeta_i \cdot \nabla \big(\partial_i \phi \big) \right) dy
\\
&= \sum_{i=1}^n \bigg[ - \int_{Z_R} \delta_{in} \psi'(y_n) \nabla \zeta_i \cdot \nabla \phi dy + \int_{\partial Z_R} \psi(y_n) \nabla \zeta_i \cdot \nabla \phi \nu_i d\sigma
\\
&\hspace{3em} + \int_{Z_R} \psi(y_n)\Delta \zeta_i \partial_i \phi dy -  \int_{\partial Z_R} \psi(y_n) \nabla \zeta_i \cdot \nu \partial_i \phi dy
\\
&\hspace{3em} + \int_{Z_R} \psi'(y_n)e_n \cdot\nabla \zeta_i   \partial_i \phi dy\bigg]
\end{align*}
The boundary terms vanish, due to the periodicity of $\zeta_i$ and $\phi$, as well as the compact support of $\psi $ in $(-R,R)$. Using $\Delta \zeta_i = h_i$ and $\nabla \cdot h = 0$, we get \begin{align*}
\int_{Z_R} \psi(y_n) \nabla v \cdot \nabla \phi dy &= - \int_{Z_R} \psi'(y_n) \big(\nabla \zeta_n \cdot \nabla \phi  + h_n \phi \big) dy 
\\
& \hspace{2em} +   \int_{Z_R} h \cdot \nabla \big(\phi \cdot \psi\big) dy + \sum_{i=1}^n \int_{Z_R} \psi'(y_n) e_n \cdot \nabla \zeta_i \partial_i \phi dy  
\\
&= \int_{Z_R} \psi'(y_n) \left( \sum_{i=1}^n \partial_n \zeta_i \partial_i \phi- \nabla \zeta_n \cdot \nabla \phi - h_i \phi \right) dy =: A_R(\phi).
\end{align*}
By a density argument this result is also valid for  $\phi = v$, and we obtain for $R\to \infty$ from the monotone convergence theorem:
\begin{align*}
\big\|\nabla v \big\|^2_{L^2(\zi)} = \lim_{R\to \infty} A_R(v).
\end{align*}
It remains to show $A_R(v) \to 0$ for $R\to \infty$. We illustrate this for one term in $A_R(v)$. From the H\"older-inequality we obtain
\begin{align*}
\left|\int_{Z_R} \psi'(y_n) h_i v dy \right| &= \left|\int_{Z_R \setminus Z_{R-1}} \psi'(y_n) h_i v dy  \right| 
\\
& \le C \|h_i\|_{L^2(Z_R \setminus Z_{R-1})} \|v\|_{L^2(Z_R \setminus Z_{R-1})} \overset{R \to \infty}{\longrightarrow} 0.
\end{align*}
This proves the claim.
\end{proof}

\begin{lemma}\label{ErrorEstimateDeltaTpmblTMtwo}
Let $c_0^M \in L^2((0,T),H^2(\Sigma))$ and $c_0^{\pm } \in L^2((0,T),H^2(\Omega^{\pm}))$. Then it holds that
\begin{align*}
\Delta_{\epsilon,T^{M,2}} + \sum_{\pm} \Delta_{\epsilon,T^{\pm,\bl}} \le C \bigg(\sqrt{\epsilon} \|c_0^M\|_{H^2(\Sigma)} \big\|\nabla \pem\big\|_{L^2(\oem)} + \epsilon \sum_{\pm} \|c_0^{\pm} \|_{H^2(\Omega^{\pm})} \big\|\nabla \pepm \big\|_{L^2(\Omega^{\pm})}\bigg)
\end{align*}
\end{lemma}
\begin{proof}
We proceed in a similar way as in Lemma \ref{ErrorEstimateDeltaTMOne}, whereby we now use Lemma \ref{SkewSymmetricTensorBoundaryLayer}. Here, the crucial point is to control the boundary terms on $\Sigma$, which occur by replacing $T^{M,2}$ and $T^{\pm,\bl}$ by skew-symmetric tensors and integrating by parts. We handle this by constructing skew-symmetric tensors which are continuous across $S^{\pm}$.

First of all, we fix $i\in \{1,\ldots,n-1\}$ and define for $j=1,\ldots,n$
\begin{align*}
h_j(y):= \begin{cases}
h_j^+(y) := T_{ji}^{+,\bl}(y - e_n) &\mbox{ for } y \in Y^+ + e_n,
\\
h_j^M(y) := T_{ji}^{M,2}(y) &\mbox{ for } y \in Z,
\\
h_j^-(y) := T_{ji}^{-,\bl}(y+e_n) &\mbox{ for } y \in Y^- - e_n.  
\end{cases}
\end{align*}
We show that $h$ fulfills the conditions of Lemma \ref{SkewSymmetricTensorBoundaryLayer} except the mean value condition. The properties of $\nabla w_{i,1}^{\pm,\bl}$ (see Lemma \ref{BoundaryLayerBulkExistence})  imply the integrability conditions on $h=(h_1,\ldots,h_n)$ from Lemma \ref{SkewSymmetricTensorBoundaryLayer} (remember that we can choose $p$ arbitrary large, especially $p>n$). Further, according to $\eqref{CellProblemSecondOrderLayerM}$, we have $\nabla \cdot h^M = 0$ in $Z$ and by $\eqref{BoundaryLayerBulk}$ we have $\nabla \cdot h^{\pm} =0$ in $Y^{\pm}$. The Neumann-boundary condition of $w_{i,2}^M$ on $S^{\pm}$ implies $h^M\cdot \nu = h^{\pm} \cdot \nu $ on $S^{\pm}$, where $\nu $ denotes the outer unit normal on $\partial Z$ with respect to $Z$. Hence, we obtain $\nabla \cdot h = 0$ on $\zi$. Now we show, that $h_n$ fulfills the mean value condition:
%

For every $a,b \in \R$ with $a<b$ and $c\in \R$ it holds that
\begin{align*}
\int_{Y^{n-1}} \int_a^b h_n dy &= \int_{Y^{n-1}} \int_a^b h \cdot \nabla (y_n + c) dy 
\\
&= b \int_{Y^{n-1} }  h_n(\y,b) d\y - a \int_{Y^{n-1} } h_n(\y,a) d\y 
\\
&\hspace{2em}+ c\left(\int_{Y^{n-1}} h_n(\y,b) d\y - \int_{Y^{n-1}} h_n(\y,a) d\y \right).
\end{align*}
We emphasize that the trace of $h_n$ exists, due to the regularity of $w_{i,1}^{\pm,\bl}$ and $w_{i,2}^M$. Since $c \in \R$ is arbitrary, we obtain 
\begin{align*}
\int_{Y^{n-1}} h_n(\y,a)d\y = \int_{Y^{n-1}} h_n(\y,b)d\y.
\end{align*}
Let us check that this term is equal to zero. Let $R > 1$ (the case $R<-1$ follows the same lines) and  $\phi \in C_{0,\#}^{\infty}(\omega(R-1,R+1))$ with $0 \le \phi \le 1$, $\|\nabla \phi\|_{\infty} \le 2$, and $\phi = 1$ in an neighborhood of $Y^{n-1} \times \{R\}$. 
From the exponential decay of $h$ we immediately obtain
\begin{align*}
\int_{\omega(R-1,R)} h \cdot \nabla \phi dy \overset{R\to \infty}{\longrightarrow}0.
\end{align*}
Since $h$ is divergence-free, we obtain
\begin{align*}
\int_{Y^{n-1}} h_n(\y,R) d\y = \int_{\omega(R-1,R)} h \cdot \nabla \phi dy  \overset{R \to \infty}{\longrightarrow} 0.
\end{align*}
This implies $\int_{Y^{n-1}} h_n(\y,a) d\y = 0$ for all $a \in \R$. Especially, we obtain $\int_{\zi} h_n dy = 0$.

For arbitrary $H \in \R^{n-1} \times \{0\}$ we define $\tilde{h}$ by $\tilde{h}^M:= h^M + H$ in $Z$ and $\tilde{h}^{\pm}:= h^{\pm} $ in $ Y^{\pm} \mp e_n$. Then, for $\tilde{h}$ we still have $\nabla \cdot \tilde{h} = 0$ in $\zi$ and $\tilde{h}$ fulfills the same integrability conditions as $h$. We choose $H$ in such a way that $\int_{\zi} \tilde{h}_j dy = 0$ for $j=1,\ldots,n-1$. Then, Lemma \ref{SkewSymmetricTensorBoundaryLayer} implies the existence of a skew-symmetric tensor $\big(\beta_{jl}^i\big) \in W^{1,p}_{1,\#}(\zi)$ with $\sum_{l=1}^n \partial_l \beta_{jl}^i = \tilde{h}_j$. 

Now, we consider the error terms $\Delta_{\epsilon, T^{M,2}} $ and $\Delta_{\epsilon,T^{\pm,\bl}}$. We  have
\begin{align*}
\Delta_{\epsilon,T^{M,2}} &= \sum_{i=1}^{n-1} \sum_{k,l=1}^n \int_{\oem} \partial_{y_l} \beta_{kl}^i \left(\fxe\right) \partial_{x_i} c_0^M \partial_{x_k} \pem dx - \sum_{i,k=1}^{n-1}  H_k \int_{\oem} \partial_{x_i} c_0^M \partial_{x_k} \pem dx
\\
&=: A_{\epsilon} + B_{\epsilon}.
\end{align*}
For the second term we immediately obtain from the H\"older-inequality
\begin{align*}
B_{\epsilon} \le C \sqrt{\epsilon} \|c_0^M\|_{H^1(\Sigma)} \big\|\nabla \pem \big\|_{L^2(\oem)} .
\end{align*}
For the first term $A_{\epsilon}$ we obtain by integration by parts
\begin{align*}
A_{\epsilon} &= \epsilon \sum_{i=1}^{n-1} \sum_{k,l=1}^n \int_{\oem} \partial_{x_l} \beta_{kl}^i \left(\fxe\right) \partial_{x_i} c_0^M \partial_{x_k} \pem dx 
\\
&= - \epsilon \sum_{i=1}^{n-1} \sum_{k,l=1}^n \int_{\oem} \beta_{kl}^i \left(\fxe\right) \left[\partial_{x_l x_i } c_0^M \partial_{x_k} \pem + \partial_{x_i } c_0^M \partial_{x_k x_l } \pem \right] dx
\\
&\hspace{2em} + \epsilon   \sum_{i=1}^{n-1} \sum_{k,l=1}^n \int_{\partial \oem} \beta_{kl}^i \left(\fxe\right) \nu_l\partial_{x_i} c_0^M \partial_{x_k} \pem d\sigma
\\
&= -\epsilon \sum_{i=1}^{n-1} \sum_{k,l=1}^n \int_{\oem} \beta_{kl}^i\left(\fxe\right) \partial_{x_l x_i } c_0^M \partial_{x_k } \pem dx 
\\
&\hspace{2em} + \sum_{\pm} \sum_{i=1}^{n-1} \sum_{k=1}^n \int_{\sepm} \pm \beta_{kn}^i \left(\fxe\right) \partial_{x_i} c_0^M \partial_{x_k} \pem dx ,
\end{align*}
where the lateral boundary terms of $\partial \oem$ vanish, due to the $\Sigma$-periodicity of the functions, and the terms involving $\partial_{x_k x_l} \pem$ vanish, due to the skew-symmetry of $\beta_{kl}^i$. With similar arguments, we obtain
\begin{align*}
\Delta_{\epsilon,T^{\pm,\bl}} &=\sum_{\pm} \sum_{i=1}^{n-1} \sum_{k,l=1}^n \int_{\Omega^{\pm}} \psi(x_n) \partial_{y_l} \beta_{kl}^i \left(\fxe \pm e_n \right) \partial_{x_i} c_0^{\pm} \partial_{x_k} \pepm dx
\\
&= - \epsilon \sum_{\pm} \sum_{i=1}^{n-1} \sum_{k,l=1}^n \int_{\Omega^{\pm}}\beta_{kl}^i \left(\fxe \pm e_n \right) \left[ \partial_{x_l x_i} c_0^{\pm} \psi(x_n) + \delta_{ln} \psi^{\prime} (x_n) \partial_{x_i} c_0^{\pm} \right] \partial_{x_k} \pepm dx 
\\
&\hspace{1em} + \epsilon \sum_{\pm} \sum_{i=1}^{n-1} \sum_{k=1}^n \int_{\Sigma} \mp \beta_{kn}^i \left(\fxe \pm e_n \right) \partial_{x_i} c_0^{\pm} \partial_{x_k} \pepm d\sigma.
\end{align*}
Adding up these terms, the boundary terms cancel out and we obtain
\begin{align*}
\Delta_{\epsilon,T^{M,2}}& + \sum_{\pm} \Delta_{\epsilon,T^{\pm,\bl}} = B_{\epsilon} - \epsilon \sum_{i=1}^{n-1} \sum_{k,l=1}^n \int_{\oem} \beta_{kl}^i\left(\fxe\right) \partial_{x_l x_i } c_0^M \partial_{x_k } \pem dx 
\\
&-\epsilon \sum_{\pm} \sum_{i=1}^{n-1} \sum_{k,l=1}^n \int_{\Omega^{\pm}}\beta_{kl}^i \left(\fxe \pm e_n \right) \left[ \partial_{x_l x_i} c_0^{\pm} \psi(x_n) + \delta_{ln} \psi^{\prime} (x_n) \partial_{x_i} c_0^{\pm} \right] \partial_{x_k} \pepm dx. 
\end{align*}
Now, using the essential boundedness of $\beta_{kl}^i$ from Lemma \ref{SkewSymmetricTensorBoundaryLayer}, we obtain the desired result.
\end{proof}

We summarize our results in the following proposition:

\begin{proposition}\label{ZusammenfassungAbschaetzungFehlertermeZweiterOrdnung}
For $c_0^{\pm} \in L^2((0,T),H^2(\Omega^{\pm}))$ with $\partial_t c_0^{\pm} \in H^1((0,T),H^1(\Omega^{\pm}))$ and $c_0^M\in L^2((0,T),H^2(\Sigma))$ with $\partial_t c_0^M \in H^1((0,T),H^1(\Sigma))$ the following estimate is valid for all $\pepm \in H^1(\Omega^{\pm})$ and $\pem \in H^1(\oem)$ with $\pepm|_{\Sigma} = \pem|_{\sepm}$:
\begin{align}
\label{SummaryEstimatesErrorTermsProp}
\begin{aligned}
\sum_{\pm}& \left[ \int_{\Omega^{\pm}} \partial_t \big(\cepm - c_{\epsilon,\app,2}^{\pm}\big) \pepm dx + \int_{\Omega^{\pm}} D^{\pm}\nabla \big(\cepm - c_{\epsilon,\app,2}^{\pm} \big) \cdot \nabla \pepm dx  \right]
\\
&+ \foe \int_{\oem} \partial_t \big(\cem - c_{\epsilon,\app,1}^M \big) \pem dx + \foe \int_{\oem} D^M\left(\fxe\right) \nabla \big(\cem - c_{\epsilon,\app,2}^M \big)\cdot \nabla \pem dx 
\\
&\le C \sum_{\pm} \left[ \|\cepm - c_0^{\pm} \|_{L^2(\Omega^{\pm})} \|\pepm\|_{L^2(\Omega^{\pm})} + \epsilon \left( 1+ \|\partial_t c_0^{\pm}\|_{H^1(\Sigma)} + \|c_0^{\pm}\|_{H^2(\Omega^{\pm})} \right) \|\pepm\|_{H^1(\Omega^{\pm})}\right]
\\
& + \frac{C}{\epsilon} \|\cem - c_0^M \|_{L^2(\oem)} \|\pem\|_{L^2(\oem)} 
\\
&+ C\sqrt{\epsilon} \|\pem\|_{H^1(\oem)} \left(1 + \|\partial_t c_0^M\|_{H^1(\Sigma)} +  \|c_0^M\|_{H^2(\Sigma)} + \sum_{\pm} \|c_0^{\pm} \|_{H^2(\Omega^{\pm})} \right).
\end{aligned}
\end{align}
\end{proposition}

\begin{proof}
For smooth functions $\pepm$ and $\pem$ the result follows directly from $\eqref{StartingEquationErrorEstimates}$ and $\eqref{AuxiliaryEquation}$, Proposition \ref{EstimateNonlinearTermsLayer} and \ref{EstimateNonlinearTermsBulk}, and Lemma \ref{EstimateDeltaEpsPartialT}, \ref{EstimateDeltaEpsRest}, \ref{ErrorEstimateDeltaTMOne}, \ref{ErrorEstimateDeltaTpmblTMtwo}, and \ref{ErrorEstimateNormalFlux}. Then, the result for functions $\pepm \in H^1(\Omega^{\pm})$ and $\pem \in H^1(\oem)$ follows by a density argument.
\end{proof}

\subsection{Error estimates for the first order approximation}
In this subsection, we give the proof of Theorem \ref{MainTheoremFirstOrderApproximation}, \ie we proof the estimate for the  error $\ce - c_{\epsilon,\app,1}$. We start from equation $\eqref{StartingEquationErrorEstimates}$ with $j=1$ and use the same methods as for  equation $\eqref{AuxiliaryEquation}$. Then for   all $\pepm \in H^1(\Omega^{\pm})$ and $\pem \in H^1(\oem)$ with $\pepm|_{\Sigma} = \pem|_{\sepm}$, we get
\begin{align}
\begin{aligned}\label{AuxialiaryEquationFirstOrder}
\sum_{\pm}& \left[ \int_{\Omega^{\pm}} \partial_t \big(\cepm - c_0^{\pm}\big) \pepm dx + \int_{\Omega^{\pm}} D^{\pm}\nabla \big(\cepm - c_{\epsilon,\app,1}^{\pm} \big) \cdot \nabla \pepm dx  \right]
\\
&+ \foe \int_{\oem} \partial_t \big(\cem - c_0^M \big) \pem dx + \foe \int_{\oem} D^M\left(\fxe\right) \nabla \big(\cem - c_{\epsilon,\app,1}^M \big)\cdot \nabla \pem dx 
\\
&= \sum_{\pm} \left[ \int_{\Omega^{\pm}} \fepm(\cepm) \pepm dx - \int_{\Omega^{\pm}} \int_{Y^n} f^{\pm}(t,y,c_0^{\pm}) \pepm dy dx  \right]
\\
&+ \foe \int_{\oem} \gem(\cem)\pem dx - \int_{\Sigma} \int_Z g^M(t,y,c_0^M) \bpem dy d\x 
\\
& +\sum_{\pm} \int_{\Sigma} D^{\pm} \nabla c_0^{\pm}\cdot \nu^{\pm} \big( \pepm - \bpem \big) d\sigma +  \Delta_{\epsilon,T^{M,1}} + \Delta_{\epsilon,rest,1}
\\
&=: \Delta_{\epsilon,1}.
\end{aligned}
\end{align}
with 
\begin{align*}
\Delta_{\epsilon,rest,1}&:= - \sum_{j=1}^{n-1} \int_{\oem} D^M\left(\fxe\right) \nabla_{\x} \partial_{x_j} c_0^M \cdot \nabla \pem w_{j,1}^M\left(\fxe\right) dx. 
\end{align*}

\begin{lemma}\label{EstimatesCorrector}
Let $c_0^{\pm} \in L^2((0,T),H^2(\Omega^{\pm}))$, then for the first order corrector $c_1^{\pm,\bl}$ in the bulk-domains it holds that
\begin{align*}
\left\|c_1^{\pm,\bl}\left(\cdot,\frac{\cdot}{\epsilon}\right)\right\|_{L^2(\Omega^{\pm})} + \epsilon \left\|\nabla c_1^{\pm,\bl}\left(\cdot,\frac{\cdot}{\epsilon}\right)\right\|_{L^2(\Omega^{\pm})} &\le C\|c_0^{\pm}\|_{H^2(\Omega^{\pm})}.
\end{align*}
Especially, we have
\begin{align*}
\|\cepm - c_{\epsilon,\app,2}^{\pm}\|_{L^2((0,T),H^1(\Omega^{\pm}))} \le C \left(1 + \|c_0^{\pm}\|_{L^2((0,T),H^2(\Omega^{\pm}))}\right)
\end{align*}
with $c_{\epsilon,\app,2}^{\pm} = c_0^{\pm} + \epsilon c_1^{\pm,\bl}\left(\cdot , \frac{\cdot}{\epsilon}\right)$.

For $c_0^M \in L^2((0,T),H^2(\Sigma))$, the first order corrector $c_1^M$ in the thin layer fulfills
\begin{align*}
\frac{1}{\sqrt{\epsilon}}\left\|c_1^M\left(\bar{\cdot} , \frac{\cdot}{\epsilon}\right)\right\|_{L^2(\oem)} 
+ \sqrt{\epsilon} \left\|\nabla c_1^M\left(\bar{\cdot} , \frac{\cdot}{\epsilon}\right)\right\|_{L^2(\oem)} &\le C \|c_0^M\|_{H^2(\Sigma)}.
\end{align*}
Especially, we have
\begin{align*}
\|\cem - c_{\epsilon,\app,1}^M\|_{L^2((0,T),H^1(\oem))} \le  C\sqrt{\epsilon} \left(1 + \|c_0^M\|_{L^2((0,T),H^2(\Sigma))}\right),
\end{align*}
with $c_{\epsilon,\app,1}^M = c_0^M + \epsilon c_1^M\left(\bar{\cdot},\frac{\cdot}{\epsilon}\right)$.
\end{lemma}
\begin{proof}
These estimates easily follow from the regularity results for $w_{j,1}^{\pm,\bl}$ and $w_{j,1}^M$ in Lemma \ref{BoundaryLayerBulk} and \ref{CellProblemFirstOrderLayer}, and the a priori estimates in Proposition \ref{ExistenceAprioriMicroscopicProblem}.
\end{proof}

\begin{theorem}\label{GeneralErrorEstimateFirstOrderApproximation}
Let $c_0^{\pm} \in L^2((0,T),H^2(\Omega^{\pm}))$ with $\partial_t c_0^{\pm} \in L^2((0,T),L^2(\Omega^{\pm}))$ and $\nabla c_0^{\pm} \in L^{\infty}((0,T)\times \Omega^{\pm})$, and $c_0^M \in L^2((0,T),H^2(\Sigma))$ with $\partial_t c_0^M \in L^2((0,T),L^2(\Sigma))$. Then, the following error estimate is valid 
\begin{align*}
\sum_{\pm} &\bigg[ \|\cepm - c_0^{\pm} \|_{L^{\infty}((0,T),L^2(\Omega^{\pm}))} + \big\|\nabla (\cepm - c_0^{\pm}) \big\|_{L^2((0,T),L^2(\Omega^{\pm}))} \bigg]
\\
+& \frac{1}{\sqrt{\epsilon}} \|\cem - c_0^M\|_{L^{\infty}((0,T),L^2(\oem))} + \frac{1}{\sqrt{\epsilon}} \big\|\nabla (\cem - c_{\epsilon,\app , 1}^M ) \big\|_{L^2((0,T),L^2(\oem))}
\\
&\le C \sqrt{\epsilon} \bigg( 1 + \|c_0^M\|_{L^2((0,T),H^2(\Sigma))} + \|\partial_t c_0^M \|_{L^2((0,T),L^2(\Sigma))} \\
& \hspace{2em} + \sum_{\pm} \left[ \|c_0^{\pm}\|_{L^2((0,T),H^2(\Omega^{\pm}))} + \|\partial_t c_0^{\pm} \|_{L^2((0,T),L^2(\Omega^{\pm}))} + \big\|\nabla c_0^{\pm}\big\|_{L^{\infty}((0,T)\times \Omega^{\pm})} \right] \bigg).
\end{align*}
\end{theorem}
\begin{proof}
We already mentioned that $\ce  - c_{\epsilon,\app,1}$ is not an admissible test function for $\eqref{StartingEquationErrorEstimates}$, hence, we add the corrector term $\epsilon c_1^{\pm} \left(x,\fxe\right)$ in the bulk-domains to obtain the admissible test function
\begin{align*}
\pepm(x)&:= \cepm - c_{\epsilon,\app,2}^{\pm} = \cepm - c_0^{\pm} - \epsilon c_1^{\pm,\bl}\left(x,\fxe\right) &\mbox{ in }& \Omega^{\pm},
\\
\pem(x)&:= \cem - c_{\epsilon,\app,1}^M = \cem - c_0^M - \epsilon c_1^M\left(\x,\fxe\right) &\mbox{ in }& \oem.
\end{align*}
We use this test function in the equation $\eqref{AuxialiaryEquationFirstOrder}$ and obtain
\begin{align*}
\sum_{\pm}& \left[\frac12 \frac{d}{dt} \|\cepm - c_0^{\pm} \|^2_{L^2(\Omega^{\pm})} + \int_{\Omega^{\pm}} D^{\pm} \nabla (\cepm - c_0^{\pm})\cdot \nabla (\cepm - c_0^{\pm}) dx\right]
\\
&+ \frac{1}{2\epsilon}\frac{d}{dt} \|\cem - c_0^M\|_{L^2(\oem)}^2 + \foe \int_{\oem} D^M\left(\fxe\right) \nabla (\cem - c_{\epsilon,\app,1}^M) \cdot \nabla (\cem - c_{\epsilon,\app,1}^M) dx 
\\
=& \Delta_{\epsilon,1} + \foe \int_{\oem } \partial_t (\cem - c_0^M) \epsilon c_1^M\left(\x,\fxe\right) dx 
\\
&+ \sum_{\pm} \left[ \int_{\Omega^{\pm}} \partial_t (\cepm - c_0^{\pm}) \epsilon c_1^{\pm,\bl}\left(x,\fxe\right) dx + \epsilon \int_{\Omega^{\pm}} D^{\pm} \nabla (\cepm - c_0^{\pm} ) \cdot \nabla c_1^{\pm,\bl}\left(x,\fxe\right) dx \right]
\\
=&:\Delta_{\epsilon,1} + C_{\epsilon} + \sum_{\pm}\left[A_{\epsilon}^{\pm} + B_{\epsilon}^{\pm} \right].
\end{align*} 
Using the coercivity of $D^{\pm}$ and $D^M$, we get for a constant $c_0>0$
\begin{align}
\begin{aligned}\label{GeneralErrorEstimateFirstOrderAuxiliaryEstimate}
\sum_{\pm} & \left[ \frac12 \frac{d}{ dt} \|\cepm - c_0^{\pm}\|^2_{L^2(\Omega^{\pm})} + c_0 \big\|\nabla (\cepm - c_0^{\pm} )\big\|^2_{L^2(\Omega^{\pm})} \right]
\\
+& \frac{1}{2\epsilon}\frac{d}{dt} \|\cem - c_0^M\|^2_{L^2(\oem)} + \frac{c_0}{\epsilon} \big\|\nabla (\cem - c_{\epsilon,\app,1}^M)\big\|^2_{L^2(\oem)} 
\\
&\le  \Delta_{\epsilon,1} + C_{\epsilon} + \sum_{\pm}\left[A_{\epsilon}^{\pm} + B_{\epsilon}^{\pm} \right] .
\end{aligned}
\end{align}
From Proposition \ref{EstimateNonlinearTermsLayer} and \ref{EstimateNonlinearTermsBulk}, and Lemma \ref{EstimateDeltaEpsRest}, \ref{ErrorEstimateNormalFlux},  and \ref{ErrorEstimateDeltaTMOne}  we obtain
\begin{align*}
\Delta_{\epsilon,1}\le& C \sum_{\pm} \big[ \|\cepm - c_0^{\pm} \|_{L^2(\Omega^{\pm})}  \|\cepm - c_{\epsilon,\app,2}^{\pm}\|_{L^2(\Omega^{\pm})} 
\\
& \hspace{3em}+ \epsilon \left(1 + \|c_0^{\pm}\|_{H^1(\Omega^{\pm})} \right) \|\cepm - c_{\epsilon,\app,2}^{\pm}\|_{H^1(\Omega^{\pm})} \big]
\\
&+ \frac{C}{\epsilon} \|\cem - c_0^M\|_{L^2(\oem)} \|\cem - c_{\epsilon,\app,1}^M\|_{L^2(\oem)} 
\\
&+   C \sqrt{\epsilon} \left(1 + \|c_0^M\|_{H^2(\Sigma)} + \sum_{\pm} \|c_0^{\pm}\|_{H^2(\Omega^{\pm})} \right) \|\cem - c_{\epsilon,\app,1}^M\|_{H^1(\oem)}
\\
\le& C \sum_{\pm} \bigg[ \big[ \|\cepm - c_0^{\pm} \|^2_{L^2(\Omega^{\pm})} + \epsilon^2 \left\|c_1^{\pm,\bl} \left(\cdot, \frac{\cdot}{\epsilon}\right)\right\|^2_{L^2(\Omega^{\pm})}
\\
& \hspace{3em}+ \epsilon \big( 1 + \|c_0^{\pm}\|_{H^1(\Omega^{\pm})} \big) \big\|\cepm - c_{\epsilon,\app,2}^{\pm}\|_{H^1(\Omega^{\pm})} \bigg)
\\
&+ \frac{C}{\epsilon} \big\|\cem - c_0^M\big\|^2_{L^2(\oem)} + \epsilon \left\|c_1^M \left(\bar{\cdot},\frac{\cdot}{\epsilon}\right)\right\|^2_{L^2(\oem)}
\\
&+ C\sqrt{\epsilon} \left(1 + \|c_0^M\|_{H^2(\Sigma)} + \sum_{\pm} \|c_0^{\pm}\|_{H^2(\Omega^{\pm})} \right) \|\cem - c_{\epsilon,\app,1}^M\|_{H^1(\oem)}.
\end{align*}
Integration with respect to time and Lemma \ref{EstimatesCorrector} yields for almost every $t \in (0,T)$
\begin{align*}
\int_0^t \Delta_{\epsilon,1} dt \le& \frac{C}{\epsilon} \big\|\cem - c_0^M \big\|^2_{L^2((0,t)\times \oem)} + C \sum_{\pm} \big\|\cepm - c_0^{\pm}\big\|^2_{L^2((0,t)\times \Omega^{\pm})}
\\
&+ \epsilon\left( 1 + \|c_0^M\|^2_{L^2((0,T),H^2(\Sigma))} + \sum_{\pm} \|c_0^{\pm}\|^2_{L^2((0,T),H^2(\Omega^{\pm}))}\right). 
\end{align*}
For $A_{\epsilon}^{\pm}$ and $C_{\epsilon}$, we immediately obtain from the H\"older-inequality
\begin{align*}
A_{\epsilon}^{\pm} &\le C\epsilon \|\partial_t (\cepm - c_0^{\pm})\|_{L^2(\Omega^{\pm})} \|c_0^{\pm}\|_{H^1(\Omega^{\pm})},
\\
C_{\epsilon} &\le C\sqrt{\epsilon} \|\partial_t (\cem - c_0^M) \|_{L^2(\oem)} \|c_0^M\|_{H^1(\Sigma)}.
\end{align*}
For $B_{\epsilon}^{\pm}$ we obtain
\begin{align*}
B_{\epsilon}^{\pm}&= \epsilon \sum_{j=1}^{n-1} \int_{\Omega^{\pm}} D^{\pm} \nabla \big(\cepm - c_0^{\pm}) \cdot \bigg[ \psi w_{j,1}^{\pm,\bl}\left(\fxe\right) \nabla \partial_{x_j} c_0^{\pm}
\\ & \hspace{7em}+ e_n \psi' \partial_{x_j} c_0^{\pm} w_{j,1}^{\pm,\bl}\left(\fxe\right) + \foe \psi \partial_{x_j} c_0^{\pm} \nabla_y w_{j,1}^{\pm,\bl}\left(\fxe\right)\bigg] dx 
\\
&\le  C \epsilon \big\|\nabla (\cepm - c_0^{\pm}) \|_{L^2(\Omega^{\pm})} \|c_0^{\pm} \|_{H^2(\Omega^{\pm})}
\\
& \hspace{4em} + C\sum_{j=1}^{n-1} \big\|\nabla (\cepm - c_0^{\pm})\big\|_{L^2(\Omega^{\pm})} \left\|\nabla_y w_{j,1}^{\pm,\bl}\left(\frac{\cdot}{\epsilon}\right)\right\|_{L^2(\Omega^{\pm})} \big\|\nabla_{\x} c_0^{\pm}\big\|_{L^{\infty}(\Omega^{\pm})}
\end{align*}
By a change of variables and the $Y^{n-1}$-periodicity of $w_{j,1}^{\pm,\bl}$, we obtain
\begin{align*}
\left\|\nabla_y w_{j,1}^{+,\bl}\left(\frac{\cdot}{\epsilon}\right)\right\|_{L^2(\Omega^{+})}^2 &\le C \int_{Y^{n-1} \times (0,H)} \left| \nabla_y w_{j,1}^{+,\bl}\left(\y,\frac{x_n}{\epsilon}\right)\right|^2 d\y dx_n
\\
&\le C\epsilon \int_{Y^+} \left| \nabla_y w_{j,1}^{+,\bl}(y) \right|^2 dy \le C\epsilon.
\end{align*}
The term including $\nabla_y w_{j,1}^{-,\bl}$ can be estimated in the same way.  Now, for every $\theta>0$ there exists a constant $C(\theta)>0$, such that for all $a,b\geq 0$ it holds that $ab \le C(\theta)a^2 + \theta b^2$. This implies 
\begin{align*}
B_{\epsilon}^{\pm} \le C(\theta) \epsilon \left(\big\|\nabla c_0^{\pm}\big\|_{L^{\infty}(\Omega^{\pm})} + \|c_0^{\pm}\|_{H^2(\Omega^{\pm})}^2 \right) + \theta \big\|\nabla (\cepm - c_0^{\pm})\big\|^2_{L^2(\Omega^{\pm})}.
\end{align*}
For $\theta< c_0$, the last term on the right-hand side can be absorbed from the left-hand side in $\eqref{GeneralErrorEstimateFirstOrderAuxiliaryEstimate}$. Integration of $\eqref{GeneralErrorEstimateFirstOrderAuxiliaryEstimate}$ with respect to time, Gronwall-inequality, and Lemma \ref{EstimatesCorrector} imply 
\begin{align*}
\sum_{\pm} &\bigg[ \|\cepm - c_0^{\pm} \|^2_{L^{\infty}((0,T),L^2(\Omega^{\pm}))} + \big\|\nabla (\cepm - c_0^{\pm}) \big\|^2_{L^2((0,T),L^2(\Omega^{\pm}))} \bigg]
\\
+& \foe \|\cem - c_0^M\|^2_{L^{\infty}((0,T),L^2(\oem))} + \foe \big\|\nabla (\cem - c_{\epsilon,\app , 1}^M ) \big\|^2_{L^2((0,T),L^2(\oem))}
\\
&\le C \epsilon \sum_{\pm} \left[ \|\partial_t (\cepm - c_0^{\pm}) \|_{L^2((0,T),L^2(\Omega^{\pm}))} \|c_0^{\pm}\|_{L^2((0,T),H^1(\Omega^{\pm}))} \right]
\\
&+ C\sqrt{\epsilon} \|\partial_t (\cem - c_0^M) \|_{L^2((0,T),L^2(\oem))} \|c_0^M\|_{L^2((0,T),H^1(\Sigma))} 
\\
&+C \epsilon \left( 1+ \|c_0^M\|^2_{L^2((0,T),H^2(\Sigma))} + \sum_{\pm}  \left[\|c_0^{\pm}\|^2_{L^2((0,T),H^2(\Omega^{\pm}))} + \big\|\nabla c_0^{\pm}\big\|_{L^{\infty}((0,T)\times \Omega^{\pm})} \right] \right)
\\
&+ \frac{1}{\epsilon} \big\|c_{\epsilon}^{0,M} - c^{0,M} \big\|^2_{L^2(\Sigma)}+ \sum_{\pm} \big\|c_{\epsilon}^{0,\pm} - c^{0,\pm}\big\|^2_{L^2(\Omega^{\pm})} .
\end{align*}
The a priori estimates from Proposition \ref{ExistenceAprioriMicroscopicProblem} and Assumption \ref{VoraussetzungenDuenneSchichtAnfangsbedingungen} imply the desired result.
\end{proof}

As a direct consequence, we obtain Theorem \ref{MainTheoremFirstOrderApproximation}:
\begin{proof}[Proof of Theorem \ref{MainTheoremFirstOrderApproximation}]
This follows directly from Theorem \ref{GeneralErrorEstimateFirstOrderApproximation} and Lemma \ref{EstimatesCorrector}.
\end{proof}

\subsection{Error estimates for the second order approximation}
\label{SectionErrorEstimateSecondOrder}

As in the case of the first order approximation, we have the problem that $\ce - c_{\epsilon,\app,2}$ is not an admissible test function, because it is not continuous across $\sepm$ (after a shift back to the domain $\oe$). We have to add a  corrector $c_2^{\pm}\left(x,\fxe\right)$ to the bulk-approximation $c_{\epsilon,\app,2}^{\pm}$ with $c_2^{\pm} \left.\left(x,\fxe\right)\right|_{\Sigma} = c_2^M\left.\left(\x,\fxe\right)\right|_{\sepm}$. Therefore, we define
\begin{align*}
w_{j,2}^{\pm}(y):= \psi(y_n) w_{j,2}^M(\y,\pm 1),
\end{align*}
\ie we extend the trace of $w_{j,2}^M$ on $S^{\pm}$ constant to the infinite strip $Y^{\pm}$ and multiply it by the cut-off function $\psi$, which was defined in Section \ref{SectionAuxiliaryProblems}. The regularity of $w_{j,2}^M$ and the trace embedding from \cite[Remark 4, Section 2.9.1]{TriebelInterpolationSpacesFunctionSpacesDifferentialOperators} imply $w_{j,2}^{\pm} \in W^{2-\frac{1}{p},p}_{\#}(Y^{\pm})$, and the H\"older-embedding \cite[Theorem 4.6.1(e)]{TriebelInterpolationSpacesFunctionSpacesDifferentialOperators} implies that $w_{j,2}^{\pm}$ and $\nabla_y w_{j,2}^{\pm}$ are essential bounded, \ie we have
\begin{align*}
\|w_{j,2}^{\pm}\|_{W^{1,\infty}(Y^{\pm})} \le C.
\end{align*}
 Now, we define 
\begin{align*}
c_2^{\pm}(x,y):= \sum_{j=1}^{n-1} \partial_{x_j} c_0^{\pm}(x) w_{j,2}^{\pm}(y) \quad \mbox{in } \Omega^{\pm} \times Y^{\pm}.
\end{align*}

\begin{lemma}\label{EstimatesCorrectorSecondOrder}
Let $c_0^{\pm} \in L^2((0,T),H^2(\Omega^{\pm}))$ with $\partial_t c_0^{\pm} \in L^2((0,T),H^1(\Omega^{\pm}))$. Then, for the   correctors $c_1^{\pm,\bl}\left(\cdot,\frac{\cdot}{\epsilon}\right)$ and  $c_2^{\pm}\left(\cdot,\frac{\cdot}{\epsilon}\right)$ in the bulk-domains it holds that
\begin{align*}
\left\|c_2^{\pm}\left(\cdot,\frac{\cdot}{\epsilon}\right)\right\|_{L^2(\Omega^{\pm})} + \epsilon \left\|\nabla c_2^{\pm}\left(\cdot,\frac{\cdot}{\epsilon}\right) \right\|_{L^2(\Omega^{\pm})} &\le C \|c_0^{\pm}\|_{H^2(\Omega^{\pm})},
\\
\left\|\partial_t c_1^{\pm,\bl}\left(\cdot,\frac{\cdot}{\epsilon}\right)\right\|_{L^2(\Omega^{\pm})} &\le C \|\partial_t c_0^{\pm} \|_{H^1(\Omega^{\pm})}.
\end{align*}

For $c_0^M\in L^2((0,T),H^2(\Sigma))$ with $\partial_t c_0^M \in L^2((0,T),H^1(\Sigma))$, the correctors $c_1^M\left(\bar{\cdot},\frac{\cdot}{\epsilon}\right)$ and $c_2^M\left(\bar{\cdot},\frac{\cdot}{\epsilon}\right)$ in the thin layer fulfill the following inequalities
\begin{align*}
\frac{1}{\sqrt{\epsilon}}\left\|c_2^M\left(\bar{\cdot} , \frac{\cdot}{\epsilon}\right)\right\|_{L^2(\oem)} 
+ \sqrt{\epsilon} \left\|\nabla c_2^M\left(\bar{\cdot} , \frac{\cdot}{\epsilon}\right)\right\|_{L^2(\oem)} &\le C \|c_0^M\|_{H^2(\Sigma)},
\\
\left\|\partial_t c_1^M \left(\bar{\cdot},\frac{\cdot}{\epsilon}\right)\right\|_{L^2(\oem)} &\le C \sqrt{\epsilon} \|\partial_t c_0^M \|_{H^1(\Sigma)}.
\end{align*} 
\end{lemma}
\begin{proof}
We only give the proof for $\nabla c_2^{\pm} \left(\cdot ,\frac{\cdot}{\epsilon}\right)$. From the essential boundedness of $w_{j,2}^{\pm}$ and $\nabla_y w_{j,2}^{\pm}$, we immediately obtain
\begin{align*}
 \left\|\nabla c_2^{\pm}\left(\cdot,\frac{\cdot}{\epsilon}\right) \right\|_{L^2(\Omega^{\pm})} &\le \sum_{j=1}^{n-1} \left\|\nabla \partial_{x_j} c_0^{\pm} w_{j,2}^{\pm} \left(\frac{\cdot}{\epsilon}\right) + \foe \partial_{x_j} c_0^{\pm} \nabla_y w_{j,2}^{\pm} \left(\frac{\cdot}{\epsilon}\right) \right\|_{L^2(\Omega^{\pm})}
 \\
&\le \frac{C}{\epsilon} \|c_0^{\pm}\|_{H^2(\Omega^{\pm})}.
\end{align*}
\end{proof}

\begin{theorem}\label{GeneralErrorEstimateSecondOrderApproximation}
Let $c_0^{\pm} \in L^2((0,T),H^2(\Omega^{\pm}))$ with $\partial_t c_0^{\pm} \in L^2((0,T),H^1(\Omega^{\pm}))$ and $c_0^M \in L^2((0,T),H^2(\Sigma))$ with $\partial_t c_0^M \in L^2((0,T),H^1(\Sigma))$. Then, the following error estimate is valid
\begin{align*}
\sum_{\pm} &\bigg[ \|\cepm - c_{\epsilon,\app,2}^{\pm} \|_{L^{\infty}((0,T),L^2(\Omega^{\pm}))} + \big\|\nabla (\cepm - c_{\epsilon,\app , 2}^{\pm}) \big\|_{L^2((0,T),L^2(\Omega^{\pm}))} \bigg]
\\
+& \frac{1}{\sqrt{\epsilon}} \|\cem - c_{\epsilon,\app,1}^M\|_{L^{\infty}((0,T),L^2(\oem))} + \frac{1}{\sqrt{\epsilon}} \big\|\nabla (\cem - c_{\epsilon,\app , 2}^M ) \big\|_{L^2((0,T),L^2(\oem))} 
\\
&\le C \epsilon \bigg( 1 + \|c_0^M\|_{L^2((0,T),H^2(\Sigma))} + \|\partial_t c_0^M \|_{L^2((0,T),H^1(\Sigma))} + \|c^{0,M}\|_{H^1(\Sigma)} \\
& \hspace{2em} + \sum_{\pm} \left[ \|c_0^{\pm}\|_{L^2((0,T),H^2(\Omega^{\pm}))} + \|\partial_t c_0^{\pm} \|_{L^2((0,T),H^1(\Omega^{\pm}))} + \|c^{0,\pm}\|_{H^1(\Omega^{\pm})}\right] \bigg).
\end{align*}
\end{theorem}

\begin{proof}
We start from the inequality $\eqref{SummaryEstimatesErrorTermsProp}$ from Proposition \ref{ZusammenfassungAbschaetzungFehlertermeZweiterOrdnung} and choose as a test function
\begin{align*}
\pepm:=& \cepm - c_{\epsilon,\app,3}^{\pm} = \cepm - c_{\epsilon,\app,2}^{\pm} - \epsilon^2 c_2^{\pm}\left(\cdot, \frac{\cdot}{\epsilon}\right) &\mbox{ in } \Omega^{\pm},
\\
\pem:=& \cem - c_{\epsilon,\app,2}^M &\mbox{ in } \oem.
\end{align*}
Let us denote by $\Delta_{\epsilon,2}$ the terms on the right-hand side in the inequality $\eqref{SummaryEstimatesErrorTermsProp}$.  Then, the coercivity of $D^M$ and $D^{\pm}$ implies for a constant $c_0>0$ after integration with respect to time that for almost every $t\in(0,T)$ it holds that
\begin{align}
\begin{aligned}\label{GeneralErrorEstimateSecondOrderAuxiliaryEstimate}
\sum_{\pm} &\left[ \frac12  \|\cepm - c_{\epsilon,\app,2}^{\pm}\|^2_{L^2(\Omega^{\pm})} + c_0 \big\|\nabla (\cepm - c_{\epsilon,\app,2}^{\pm} \big\|^2_{L^2((0,t)\times \Omega^{\pm})}\right]
\\
&+ \frac{1}{2\epsilon} \|\cem - c_{\epsilon,\app,1}^M\|^2_{L^2(\oem)} + \frac{c_0}{\epsilon} \big\|\nabla (\cem - c_{\epsilon,\app,2}^M)\big\|^2_{L^2((0,t)\times \oem)}
\\
\le& \int_0^t \Delta_{\epsilon,2}dt + \frac{1}{2\epsilon}\|\cem(0) - c_{\epsilon,\app,1}^M(0)\|^2_{L^2(\oem)} 
+ \sum_{\pm} \left[ \frac12 \|\cepm(0) - c_{\epsilon,\app,2}^{\pm}(0)\|^2_{L^2(\Omega^{\pm})}\right]
\\
& +\foe \int_0^t\int_{\oem} \partial_t (\cem - c_{\epsilon,\app,1}^M) \epsilon^2 c_2^M\left(\x,\fxe\right)dx dt
\\
&+ \sum_{\pm} \bigg[ \int_0^t\int_{\Omega^{\pm}} \partial_t (\cepm - c_{\epsilon,\app,2}^{\pm}) \epsilon^2 c_2^{\pm}\left(x,\fxe\right) dxdt
\\ & \hspace{3em} + \epsilon^2 \int_0^t\int_{\Omega^{\pm}} D^{\pm} \nabla (\cepm - c_{\epsilon,\app,2}^{\pm}) \cdot \nabla c_2^{\pm}\left(x,\fxe\right)dx dt \bigg]
\\
=&: \Delta_{\epsilon,2,t} + \Delta_{\epsilon,0} +  C_{\epsilon} + \sum_{\pm}\left[A_{\epsilon}^{\pm} + B_{\epsilon}^{\pm}\right].
\end{aligned}
\end{align}
Here, $\Delta_{\epsilon,0}$ includes all terms containing initial values.
Let us estimate the terms on the right-hand side. Some elemental calculations and the estimates from Proposition \ref{ExistenceAprioriMicroscopicProblem} and Lemma \ref{EstimatesCorrector} and \ref{EstimatesCorrectorSecondOrder} give us (we emphasize that these equations are valid pointwise in time)
\begin{align*}
\|\cepm - c_0^{\pm}\|_{L^2(\Omega^{\pm})} \|\cepm - c_{\epsilon,\app,3}^{\pm}\|_{L^2(\Omega^{\pm})} &\le C \|\cepm - c_{\epsilon,\app,2}^{\pm}\|_{L^2(\Omega^{\pm})}^2 
\\
& \hspace{5em}+ C\epsilon^2 \left(1 + \|c_0^{\pm}\|^2_{H^1(\Omega^{\pm})}\right),
\\
\foe \|\cem - c_0^M\|_{L^2(\oem)} \|\cem - c_{\epsilon,\app,2}^M\|_{L^2(\oem)} &\le \frac{C}{\epsilon} \|\cem - c_{\epsilon,\app,1}^M\|^2_{L^2(\oem)} 
\\
 &\hspace{5em}+ C\epsilon^2 \left(1  + \|c_0^M\|^2_{L^2(\Sigma)} \right),
\\
\|\cepm - c_{\epsilon,\app,3}^{\pm}\|_{H^1(\Omega^{\pm})} \le \|\cepm - c_{\epsilon,\app,2}^{\pm}&\|_{H^1(\Omega^{\pm})} + C\epsilon \|c_0^{\pm}\|_{H^1(\Omega^{\pm})},
\\
\frac{1}{\sqrt{\epsilon}} \|\cem - c_{\epsilon,\app,2}^M\|_{H^1(\oem)} \le \frac{1}{\sqrt{\epsilon}} \|&\cem - c_{\epsilon,\app,1}^M \|_{H^1(\oem)} + C\epsilon \|c_0^M\|_{H^2(\Sigma)}.
\end{align*}
These estimates  imply (together with the inequality $ab \le C(\theta)a^2 + \theta b^2$) almost everywhere in $(0,T)$
\begin{align*}
\Delta_{\epsilon,2} \le& C \sum_{\pm} \left[ \|\cepm - c_{\epsilon,\app,2}^{\pm} \|^2_{L^2(\Omega^{\pm})} +\epsilon^2 \left( 1 + \|\partial_t c_0^{\pm}\|_{H^1(\Omega^{\pm})} + \|c_0^{\pm}\|^2_{H^1(\Omega^{\pm})} \right)  \right]
\\
&+ \frac{c_0}{4} \sum_{\pm} \left[\big\|\nabla (\cepm - c_{\epsilon,\app,2}^{\pm})\big\|^2_{L^2(\Omega^{\pm})}\right] + \frac{C}{\epsilon} \|\cem - c_{\epsilon,\app,1}^M\|^2_{L^2(\oem)} 
\\
&+C \epsilon^2 \left( 1 + \|\partial_t c_0^M\|^2_{H^1(\Sigma)} + \|c_0^M\|^2_{H^2(\Sigma)} \right) + \frac{c_0}{4\epsilon} \big\|\nabla (\cem - c_{\epsilon,\app,2}^M\big\|^2_{L^2(\oem)}.
\end{align*}
Using again Proposition \ref{ExistenceAprioriMicroscopicProblem} and Lemma \ref{EstimatesCorrector} and \ref{EstimatesCorrectorSecondOrder}, we obtain for almost every $t\in (0,T)$
\begin{align*}
A_{\epsilon}^{\pm} &\le C\epsilon^2 \left( 1 + \|\partial_t c_0^{\pm}\|^2_{L^2((0,t),H^1(\Omega{\pm}))} +  \|c_0^{\pm}\|_{L^2((0,t),H^2(\Omega^{\pm}))} \right),
\\
B_{\epsilon}^{\pm} &\le C\epsilon^2 \|c_0^{\pm}\|^2_{L^2((0,t),H^2(\Omega^{\pm}))} + \frac{c_0}{4} \big\|\nabla (\cepm - c_{\epsilon,\app,2}^{\pm} )\big\|_{L^2((0,t)\times \Omega^{\pm})}^2,
\\
C_{\epsilon} &\le C \epsilon^2 \left( 1 + \|\partial_t c_0^M\|^2_{H^1(\Sigma)} + \|c_0^M\|_{L^2((0,t),H^2(\Sigma))}^2 \right).
\end{align*} 
It remains to estimate the initial term $\Delta_{\epsilon,0}$. Assumption \ref{VoraussetzungenDuenneSchichtAnfangsbedingungen} and \ref{ZusatzbedingungAW}, the regularity of $c_0^{\pm}$ and $c_0^M$, Lemma \ref{EstimatesCorrector}, and the essential boundedness of $w_{j,1}^{\pm,\bl}$ and $w_{j,1}^M$ imply
\begin{align*}
\Delta_{\epsilon,0} \le& \frac{C}{\epsilon} \left( \left\|c_{\epsilon}^{0,M} - c^{0,M} \right\|^2_{L^2(\oem)} + \epsilon^2 \left\|c_1^M\left(0,\x,\fxe\right)\right\|^2_{L^2(\oem)} \right)
\\
&+ \sum_{\pm} \left( \left\|c_{\epsilon}^{0,\pm} - c^{0,\pm} \right\|^2_{L^2(\Omega^{\pm})} + \epsilon^2 \left\|c_1^{\pm} \left(0,x,\fxe\right)\right\|^2_{L^2(\Omega^{\pm})}\right)
\\
\le& C \epsilon^2 \left(1 + \|c^{0,M}\|^2_{H^1(\Sigma)} + \|c^{0,\pm}\|^2_{H^1(\Omega^{\pm})}\right).
\end{align*}

Plugging in these estimates in $\eqref{GeneralErrorEstimateSecondOrderAuxiliaryEstimate}$, we obtain the desired result.
\end{proof}

As an immediate consequence we obtain Theorem \ref{MainTheoremSecondOrderApproximation}.

\begin{proof}
[Proof of Theorem \ref{MainTheoremSecondOrderApproximation}]
This is a direct consequence of Theorem \ref{GeneralErrorEstimateSecondOrderApproximation} and Proposition \ref{ExistenceAprioriMicroscopicProblem}.
\end{proof}

\begin{remark}\label{BemerkungVerallgemeinerungen}\mbox{}
\begin{enumerate}
[label = (\roman*)]
\item Our results are not restricted to the Assumptions \ref{VoraussetzungenDuenneSchichtReaktionskinetikBulk} - \ref{VoraussetzungenDuenneSchichtAnfangsbedingungen}, and \ref{ZusatzbedingungFplusminus} - \ref{ZusatzbedingungAW}. In fact, the error estimates from Theorem \ref{MainTheoremFirstOrderApproximation} and \ref{MainTheoremSecondOrderApproximation} remain valid, if the microscopic solution fulfills the a priori estimates from Proposition \ref{ExistenceAprioriMicroscopicProblem}, the nonlinear functions fulfill the estimates from Proposition \ref{EstimateNonlinearTermsLayer} and \ref{EstimateNonlinearTermsBulk}, the initial values fulfill the error estimate $\eqref{ErrorEstimateInitialValues}$, and the macroscopic solution fulfills the regularity hypothesis from Theorem \ref{MainTheoremFirstOrderApproximation} and \ref{MainTheoremSecondOrderApproximation}.

\item Theorem \ref{GeneralErrorEstimateFirstOrderApproximation} and therefore Theorem \ref{MainTheoremFirstOrderApproximation} still hold if we replace the condition $\eqref{ErrorEstimateInitialValues}$ in Assumption \ref{VoraussetzungenDuenneSchichtAnfangsbedingungen} by the weaker condition
\begin{align*}
\big\|c_{\epsilon}^{0,\pm} - c^{0,\pm}\big\|_{L^2(\Omega^{\pm})} + \frac{1}{\sqrt{\epsilon}} \big\|c_{\epsilon}^{0,M} - c^{0,M} \big\|_{L^2(\oem)} \le C \sqrt{\epsilon}.
\end{align*}
This can be easily seen from the last inequality in the proof of Theorem \ref{GeneralErrorEstimateFirstOrderApproximation}. In this case the error for the second order approximation is also of order $\epsilon^{\frac12}$, \ie 
no better error estimate is obtained by using the improved approximation $c_{\epsilon,\app,2}$.
\end{enumerate}

\end{remark}

\section{Conclusion and outlook}
In our paper, we derived approximations $c_{\epsilon,\app,j}$ of order $j=1$ and $j=2$ for the microscopic solution $\ce = (\ce^+,\cem,\ce^-)$ of the microscopic model $\eqref{MicroscopicModel}$ describing nonlinear and nonstationary reaction-diffusion processes in media separated by thin periodically heterogeneous layers. The approximations are constructed via the solution of the macroscopic model $\eqref{MacroscopicProblemZeroOrder}$ and additional correctors including solution of cell problems and boundary layers. The challenging aspects are related to the simultaneous scale transition for the thickness of the layer and the periodicity within the layer, the coupling between the bulk domains and the thin layer, and the nonlinear reaction kinetics.

Using the iterative procedure of our paper, it is possible to construct approximations of higher order. However, in the general setting considered in this paper it is not possible to obtain higher order error estimates. 
First of all, the oscillations in the reaction-kinetics lead to an error of order $\epsilon$ (even if the reaction term is independent of the solution). If we consider nonlinear reaction-kinetics independent of the oscillating variable, we need additional regularity for the nonlinear functions (Taylor-expansion) to perform the asymptotic expansion. In our applications, however, see e.g., 
\cite{GahnNeussRaduKnabnerEffectiveModelSubstrateChanneling}, where Michaelis-Menten-type kinetics are required, we can only expect to have Lipschitz regularity. Last but not least, due to different geometrical structures present in the microscopic and the macroscopic model (the thin layer reduces to a lower dimensional interface), we obtain an approximation error of order $\epsilon$, see Lemma 15, 
and it is not clear how this error can be reduced further. We emphasize that in our proof the derivation of the $L^2$-error and the $H^1$-error go hand in hand, hence it is not possible to improve the $L^2$-error.

In this paper we only considered the case of high diffusion in the thin layer, i.e., the diffusion coefficient in the layer is of order $\epsilon^{-1}$. Here, the limit solution in the layer does not depend on the microscopic variable $y\in Z$. This is also the case for moderate diffusion, when $\gamma \in (-1,1)$, see \cite{GahnEffectiveTransmissionContinuous}.
However, in this case the diffusion term in the homogenized problem in the layer vanishes and the homognized solution in the layer satisfies  an ordinary differential equation. In the case of very small diffusion, when $\gamma = 1$, see \cite{NeussJaeger_EffectiveTransmission},
the macroscopic model depends on both, the microscopic variable $y\in Z$ and the macroscopic variable $x \in \Sigma$. This completely different structure of the macroscopic problem in the case $\gamma =1$ compared to the cases $\gamma \in [-1, 1)$ does not permit to derive error estimates by the methods developed in this paper directly. On the other hand, for the case $\gamma \in (-1,1)$, we see the possibility to add an additional term in the macroscopic equation at the interface $\Sigma$, namely
$$
\epsilon^{\gamma+ 1} \nabla_{x} \cdot \big(D^{M,\ast}\nabla_{x} \big),
$$
to obtain a reaction-diffusion equation on $\Sigma$, as in the case $\gamma = -1$ considered in this paper:
\begin{eqnarray*}
[D^{\pm} \nabla c_{0,\epsilon}^{\pm} \cdot \nu ]\!] &=& |Z|\partial_t c_{0,\epsilon}^M - \epsilon^{1  + \gamma}|Z|\nabla_{x} \cdot \left(D^{M,\ast}\nabla_{x} c_{0,\epsilon}^M\right) \\
&- &\int_Z g^M(t,y,c_{0,\epsilon}^M) dy\quad  \mbox{ on } (0,T)\times \Sigma.
\end{eqnarray*}
Then we can use the same methods as above and expect error estimates depending on $\gamma$, and getting worse in the critical limit $\gamma \to 1$. 
However, this is part of ongoing work.

\appendix

\section{Regularity proofs for the macroscopic solution}
\label{SectionRegularity}

In this section, we give the proof of the regularity results in Proposition \ref{RegularityResultsMacroscopicSolution} without going into detail. We shortly write:
\begin{align*}
V^{\pm}&:= L^{\infty}((0,T),L^2(\Omega^{\pm})) \cap L^2((0,T),H^1(\Omega^{\pm})),
\\
V^M&:= L^{\infty}((0,T),L^2(\Sigma)) \cap L^2((0,T),H^1(\Sigma)),
\end{align*}
and 
\begin{align*}
\|u^{\pm}\|_{V^{\pm}}&:= \|u\|_{L^{\infty}((0,T),L^2(\Omega^{\pm}))} + \|u^{\pm}\|_{L^2((0,T),H^1(\Omega^{\pm}))},
\\
\|u^M\|_{V^M} &:= \|u^M\|_{L^{\infty}((0,T),L^2(\Sigma))} + \|u^M\|_{L^2((0,T),H^1(\Sigma))}.
\end{align*}
Further we define for $M\in \R$ and $t\in (0,T)$, and given $u^{\pm} \in V^{\pm}$ and $u^M \in V^M$  the sets
\begin{align*}
\Omega_M^{\pm}(t):= \{x \in \Omega^{\pm} \, : \, u^{\pm}(t,x)>M\}, \quad \Sigma_M(t):= \{x\in \Sigma \, : \, u^M(t,x) > M \}.
\end{align*}
The following technical lemma is an extension of \cite[Chapter II, Theorem 6.1]{Ladyzenskaja} for our geometrical setting:
\begin{lemma}\label{VerallgemeinerungLady}
Let $u^{\pm} \in V^{\pm}$, $u^M\in V^M$,  $\widehat{M} \geq 0$, and assume that there exist $C >0 $ and $\kappa >0$, such that for all $M \geq \widehat{M}$ it holds that
\begin{align*}
\|(u^M-M)_+\|_{V^M} + \sum_{\pm}\|(u^{\pm}-M)_+\|_{V^{\pm}} \le C M \left( \int_0^T |\Sigma_M(t)|^{\frac{r}{q}} dt + \sum_{\pm} \int_0^T |\Omega^{\pm}_M(t)|^{\frac{r}{q}}dt \right)^{\frac{1+\kappa}{r}}
\end{align*}
for $r \in [2,\infty]$ and $q \in \left[2, \frac{2n}{n-2}\right]$ (here $\frac{1}{0}= \infty$ for $n=2$). Then, we have
\begin{align*}
u^{\pm} \in L^{\infty}((0,T)\times \Omega^{\pm}),\quad u^M \in L^{\infty}((0,T)\times \Sigma).
\end{align*}
\end{lemma}
\begin{proof}
The proof can easily be adapted from \cite[II, Theorem 6.1]{Ladyzenskaja} and is skipped here.
\end{proof}

\begin{proof}[Proof of Proposition \ref{RegularityResultsMacroscopicSolution}]

First of all, we obtain $c_0^{\pm} \in L^2((0,T),H^2(\Omega^{\pm}))$ and $c_0^M \in L^2((0,T),H^2(\Sigma))$ by testing $\eqref{VE_Omega}$ with the difference quotient of $(\nabla c_0^+,\nabla_{\x} c_0^M, \nabla c_0^-)$ with respect to the spatial variable for the $i$-th component with $i=1,\ldots,n-1$ ($c_0$ can be extended periodically to $\R^{n-1} \times (-H,H)$). For more details, we refer to \cite[Lemma 2.4]{GahnEffectiveTransmissionContinuous}.

The regularity $c_0^{\pm} \in H^1((0,T),H^1(\Omega^{\pm}))$ and $c_0^M \in H^1((0,T),H^1(\Sigma))$ is formally obtained by differentiating the problem $\eqref{MacroscopicProblemZeroOrder}$ with respect to time. Rigorously, this can be done by a priori estimates for a Galerkin approximation. For more details we refer to \cite{EvansPartialDifferentialEquations}. Here, we make use of the differentiability of $f^{\pm}$ and $g^M$ with respect to $t$ from the additional assumptions \ref{ZusatzbedingungFplusminus} and \ref{ZusatzbedingungGM}, as well as the essential boundedness of $\nabla_{\x} c^{0,M}$ and $\nabla_{\x} c^{0,\pm}$  from \ref{ZusatzbedingungAW}.

It remains to check the essential boundedness of $\nabla_{\x} c_0^{\pm}$ and $\nabla_{\x} c_0^M$. Formally this can be done by differentiating the equation$\eqref{MacroscopicProblemZeroOrder}$ with respect to $x_i$ for $i=1,\ldots,n-1$. Let us do this in a more rigorous way, where the ideas can be found in \cite{Ladyzenskaja}. As test functions in $\eqref{VE_Omega}$, we choose $\partial_{x_i} \phi^{\pm} $ and $\partial_{x_i} \phi^M$ with $\phi \in C^{\infty}_{\#}(\Omega)$ and $\phi^{\pm} = \phi|_{\Omega^{\pm}}$ and $\phi^M = \phi|_{\Sigma}$. Then, by integration by parts we obtain with the periodic boundary conditions on the lateral boundary and the regularity results from above the following weak formulation for $v_i^{\pm}:= \partial_{x_i} c_0^{\pm}$ and $v_i^M := \partial_{x_i} c_0^M$
\begin{align*}
\sum_{\pm}& \left[ \int_{\Omega^{\pm}} \partial_t v_i^{\pm} \phi^{\pm} dx + \int_{\Omega^{\pm}} D^{\pm} \nabla v_i^{\pm} \nabla \phi^{\pm} dx \right]
\\
&+ |Z|\int_{\Sigma} \partial_t v_i^M \phi^M d\x + |Z|\int_{\Sigma} D^{M,\ast} \nabla_{\x} v_i^M \nabla_{\x} \phi^M d\x 
\\
=& \int_{\Sigma}\int_Z \partial_z g^M(c_0^M) v_i^M \phi^M dy d\x + \sum_{\pm} \int_{\Omega^{\pm}} \int_{Y^n} \partial_z f^{\pm}(c_0^{\pm}) v_i^{\pm}\phi^{\pm} dy dx,
\end{align*}
together with the initial condition $v_i^{\pm}(0) = \partial_{x_i} c^{0,\pm}$ and $v_i^M(0) =\partial_{x_i} c^{0,M}$ (this is exactly the weak formulation we obtain by formally differentiating $\eqref{MacroscopicProblemZeroOrder}$ with respect to $x_i$). By density, see \cite[Lemma 5.3]{GahnEffectiveTransmissionContinuous}, this equation is valid for all $\phi \in H^1(\Omega)$ with $\phi|_{\Sigma} \in H^1(\Sigma)$. For $M\geq M_0$ (see \ref{ZusatzbedingungAW}), we choose $\phi^{\pm}:= (v_i^{\pm} - M)_+ $ and $\phi^M:= (v_i^M - M)_+$ with $(\cdot)_+:= \max\{0,\cdot\}$ and obtain by an elemental calculation using the essential boundedness of $\partial_z f^{\pm}$ and $\partial_z g^M$ (which follows from the uniform Lipschitz continuity with respect to $Z$), the estimate (for a constant $c_0>0$)
\begin{align*}
\sum_{\pm}& \left[ \frac12 \frac{d}{dt} \big\|(v_i^{\pm} - M)_+ \big\|^2_{L^2(\Omega^{\pm})} + c_0 \big\|\nabla (v_i^{\pm} - M)_+\big\|^2_{L^2(\Omega^{\pm})} \right]
\\
&+ \frac12 \frac{d}{dt} \big\|(v_i^M - M)_+ \big\|^2_{L^2(\Sigma) } + c_0 \big\|\nabla_{\x} (v_i^M - M)_+ \big\|^2_{L^2(\Sigma)}
\\
\le& C \left[ \big\|(v_i^M - M)_+\big\|^2_{L^2(\Sigma)} + \sum_{\pm} \big\|(v_i^{\pm} - M)_+ \big\|^2_{L^2(\Omega^{\pm})} \right]
\\
&+ CM \int_{\Sigma} (v_i^M - M)_+ d\x + CM\sum_{\pm} \int_{\Omega^{\pm}} (v_i^{\pm} - M)_+ dx  .
\end{align*}
Integration with respect to time, the condition \ref{ZusatzbedingungAW}, and the Gronwall inequality imply
\begin{align*}
\big\|(v_i^M - M)_+\big\|_{V^M} &+ \sum_{\pm} \big\|(v_i^{\pm} - M)_+ \big\|_{V^{\pm}}
\\
&\le C M \left( \int_0^T |\Sigma_M(t)|dt + \sum_{\pm} \int_0^T |\Omega^{\pm}_M(t)|dt \right)^{\frac{1}{2}}.
\end{align*}
An application of Lemma \ref{VerallgemeinerungLady} with $r=q = 2\frac{n+2}{n}$ and $\kappa = \frac{2}{n}$ gives the desired result.
\end{proof}

\section*{Acknowledgment}

The authors acknowledge the support by the Center for Modelling and Simulation in the Biosciences (BIOMS) of the University of Heidelberg. The first author was supported by the   project  SCIDATOS  (Scientific  Computing  for  Improved  Detection  and  Therapy  of  Sepsis), which was  funded  by  the  Klaus Tschira Foundation, Germany (Grant number 00.0277.2015).

  \bibliographystyle{abbrv} 
  \bibliography{literature}

\end{document}